\documentclass{ws-cm}

\newcommand{\R}{{\mathbb R}}

\usepackage{placeins}

\begin{document}

\markboth{Descombes, Duarte, Dumont, Louvet, Massot}{Adaptive Time Splitting Method for Multi-scale Evolutionary Partial Differential Equations}

%%%%%%%%%%%%%%%%%%% Publisher's Area please ignore %%%%%%%%%%%%%%%%%%%%%%
\catchline{}{}{}{}{}
%%%%%%%%%%%%%%%%%%%%%%%%%%%%%%%%%%%%%%%%%%%%%%%%%%%%%%%%%%%%%%%%%%%%%%%%%

\title{ADAPTIVE TIME SPLITTING METHOD FOR MULTI-SCALE 
EVOLUTIONARY PARTIAL DIFFERENTIAL EQUATIONS}

\author{ST\'EPHANE DESCOMBES}

\address{Laboratoire J. A. Dieudonn\'e -
UMR CNRS 6621,
Universit\'e Nice Sophia Antipolis,
Parc Valrose, 
Nice Cedex 02, 06108, France\\
sdescomb@unice.fr}

\author{MAX DUARTE}
\address{Laboratoire EM2C - UPR CNRS 288, Ecole Centrale
Paris, Grande Voie des Vignes \\
 Chatenay-Malabry Cedex, 92295, France \\
max.duarte@em2c.ecp.fr}

\author{THIERRY DUMONT}

\address{Institut Camille Jordan - UMR CNRS 5208,
Universit\'e de Lyon,
Universit\'e Lyon 1,
INSA de Lyon 69621, 
Ecole Centrale de Lyon,
43 Boulevard du 11 novembre 1918, \\
 Villeurbanne Cedex, 69622, France\\
tdumont@math.univ-lyon1.fr}

\author{VIOLAINE LOUVET}

\address{Institut Camille Jordan - UMR CNRS 5208,
Universit\'e de Lyon,
Universit\'e Lyon 1,
INSA de Lyon 69621, 
Ecole Centrale de Lyon,
43 Boulevard du 11 novembre 1918,\\
Villeurbanne Cedex, 69622, France,\\
louvet@math.univ-lyon1.fr}

\author{MARC MASSOT}

\address{Laboratoire EM2C - UPR CNRS 288, Ecole Centrale
Paris, Grande Voie des Vignes, \\
 Chatenay-Malabry Cedex, 92295
France,\\
marc.massot@em2c.ecp.fr}

\maketitle

\begin{history}
\received{Day Month Year}
\revised{Day Month Year}
%\accepted{(Day Month Year)}
\end{history}

\begin{abstract}
This paper
introduces an adaptive time splitting
technique 
for the solution of stiff evolutionary PDEs
that guarantees
an effective error control
of the simulation,
independent of the
fastest physical time scale
for highly unsteady
problems.
The strategy 
considers a second order Strang method and 
another lower order embedded splitting scheme
that takes into account potential loss of order due to the stiffness
featured by time-space multi-scale phenomena.
The scheme 
is then built upon
a precise numerical analysis of the method 
and a
complementary numerical procedure,
conceived
to overcome classical restrictions of adaptive time stepping
schemes 
based on lower order embedded methods,
whenever asymptotic estimates fail to predict the
dynamics of the problem.
The performance of the method in terms of control 
of integration errors is evaluated by
numerical simulations of stiff propagating waves coming
from nonlinear chemical dynamics models as well
as
highly multi-scale
nanosecond repetitively pulsed gas discharges,
which allow to illustrate the method capabilities
to consistently describe 
a broad spectrum of time scales 
and different physical scenarios
for consecutive discharge/post-discharge phases.
\end{abstract}

\keywords{Time adaptive integration; error control; operator splitting; reaction-diffusion; multi-scale reaction waves; multi-scale discharge.}

\ccode{AMS Subject Classification: 65G20, 65M15, 65Z05, 65L04, 35K57, 35A35, 35C07}

\begin{center}
\textbf{Dedication} 
\end{center}
Cet article est d\'edi\'e \`a la m\'emoire de Michelle Schatzman. 
Sp\'ecialiste des m\'ethodes de d\'ecomposition d'op\'erateur, sa grande clairvoyance
scientifique lui a permis d'orienter plusieurs chercheurs d\'ebutants  
sur ce sujet \`a un moment o\`u il pouvait sembler achev\'e. Michelle aimait dire qu'il n'y a  
pas de fronti\`ere entre les branches des math\'ematiques et que seule une grande culture
permet de naviguer dans cette for\^et et d'y trouver les bonnes  
techniques pour r\'esoudre un probl\`eme. Ce travail est un hommage; \`a la crois\'ee des  
math\'ematiques et de leur applications effectives,
il tente d'illustrer cette assertion. Michelle, ton dynamisme, ton  
humour et ton plaisir \`a parler math\'ematiques nous manquent.

\section{Introduction}
Numerical simulations of multi-scale phenomena are commonly used
for modeling purposes
in many applications such as combustion, plasma discharges,
chemical vapor deposition
or air pollution modeling.
In general, all these models raise se\-ve\-ral difficulties created by the high number of unknowns,
the wide range of temporal scales due to large and 
detailed chemical kinetic mechanisms, as well as  
steep spatial gradients 
associated with localized fronts of high chemical activity.
In this context, faced with the induced stiffness
of these time dependent problems,
a high performing numerical strategy for multidimensional simulations 
considers a 
 time operator splitting 
with dedicated high order time integration methods for reaction and diffusion problems, in order to 
exploit efficiently the special features
of each problem. 
Such a numerical strategy for time discretization has been presented 
in \cite{article_avc} and extended in  \cite{article_mr} with multiresolution  techniques
for adaptive space discretization. 
The main idea is to use a second order
Strang scheme
to solve independently
reaction and diffusion problems
in three successive fractional steps,
taking into account that for 
multi-scale phenomena
better performances are usually expected while ending
the splitting scheme by the part involving the fastest scales,
as it has been proven in \cite{Descombes04}.
Therefore, based on these theoretical results and on the construction of the splitting
solver, this strategy
provides an accurate resolution of such stiff problems
even for
splitting time steps much larger than either
the fastest time scales involved in the source terms
or the time step restrictions related to spatial grid discretizations.

Up to our days,
fixed splitting time step schemes have been 
largely used in the literature \cite{Knio99,Singer2006150,OranBoris2001},
and the relevance of our numerical strategy \cite{article_avc,article_mr}
has been evaluated in the framework of
stiff reaction waves for which a constant splitting time step is more than reasonable
to precisely describe the global coupling of the split phenomena.
However,
such a fixed time stepping strategy
would surely lead to major difficulties and limitations for problems 
describing highly non stationary models with
very different dynamics such as 
flame ignition and propagation or repetitively pulsed plasmas discharges \cite{Pilla:2006},
all the more in the framework of large scale simulations. 
It is thus essential to be able to dynamically adapt splitting time steps for the simulation of such multi-scale problems with strongly evolving dynamics.

In order to guarantee a precise description of the coupled multi-scale phenomenon,
this splitting time step adaptation strategy must rely 
on a local error 
estimate, which can be obtained by 
considering a lower order embedded method.
This is a common practice for ODEs numerical solution \cite{Hairer02}, which 
yields very efficient and eventually high order 
methods 
for which
time steps can dynamically adapt according to a given tolerance,
to sufficiently small values in order to cope with the fastest time scales of the problem.
However,
it is well known
that for stiff problems and larger accuracy tolerances, the order
of the methods can degenerate, yielding non 
reliable error estimates and possibly, much larger global errors than expected by the given tolerance. 
Such a scenario will be all the more valid in the framework of the resolution of PDEs
where fine grid and large gradients coupled with stiff source terms lead to especially stiff problems. 
In particular,
our numerical strategy \cite{article_avc,article_mr} is built in such a way that the main source of error is the 
splitting error, each building block relying on high order adaptive and dedicated numerical methods;
therefore,
it is essential not only to construct a reliable splitting error estimate,
but also to 
guarantee an effective error control within 
the so claimed accuracy 
tolerance.

In this article, we present a novel strategy to control the local splitting error with two different
splitting schemes, the first one is a second order Strang technique whereas
 the second one considers a 
shifted Strang formula,
built with a $\varepsilon$-shift in time of the classical Strang formula.
This second method 
is embedded because
the first substep is common to both methods
to reduce computational cost, and
inherits from the Strang scheme,
stability properties and the
same numerical behavior 
in the context of stiff problems;
nevertheless, 
it is only of order
one due to the slightly lack of symmetry. 
In the first part of the paper, 
we conduct a complete error estimate of 
this new splitting method 
in order to characterize the local error estimate that will be 
computed out of first and second order splitting resolutions.
We define then
a domain of application of the adaptive method in which the local error estimates
guarantee an effective error control of the solution according to the given tolerance.
The key issue is related to the evaluation of a maximum splitting time step, called the critical splitting time step, 
as a function of $\varepsilon$, for which local error estimates are valid. 
A numerical validation of the theoretical estimates is performed 
in the framework of traveling reaction waves for a simple PDE, 
for which the threshold and critical time steps can be also theoretically estimated and compared with 
numerical results.

However, 
in order to extend the numerical strategy to more realistic configurations, 
for which theoretical evaluation of critical time steps 
is out of reach, 
we develop a complementary and general numerical procedure based on 
numerical estimates,
that allows to establish the domain of application of the method
by simultaneously choosing the appropriate $\varepsilon$ for a given tolerance.
This procedure 
is tested 
in the framework of nonlinear chemical dynamics of Belousov-Zhabotinsky (BZ) reactions in a very stiff case in both time and space,
yielding satisfactory results. 
As a consequence,
a final numerical strategy is conceived 
that considers adaptive splitting time steps and 
that evaluates simultaneously critical time steps as well as best-suited $\varepsilon$, 
in order to 
guarantee error control for a given accuracy tolerance of the simulation
with splitting time steps as large as possible.
The relevance of the proposed strategy is first evaluated for the BZ reaction-diffusion equations, 
whereas 
a more complex problem issued from the simulation of multi-pulsed gas discharges
 involving several dynamics with very different typical time scales, constitutes the second test-case. 
It is shown that for this second very stiff reaction-diffusion system, 
splitting time steps can cover a range of three orders of magnitude and always guarantee a proper 
respect of the prescribed tolerance.

The paper is organized as follows:
section \ref{AdaptSec} describes
the adaptive time splitting strategy;
in section \ref{analyse},  we perform
 the numerical analysis of the proposed method
 and identify the  limit of validity of the local error estimate which is at the heart of the adapting procedure.
Section \ref{KPP} is devoted to the 
validation of the previous theoretical estimates
and to a theoretical/numerical study of the critical splitting  time steps 
in the context of a 1D reaction-diffusion
problem featuring traveling wave solutions.
In section \ref{Strategy} we present the final numerical strategy 
that includes an additional numerical procedure to evaluate critical time steps
and suitable $\varepsilon$.
The potential of the method is  illustrated for the proposed two test-cases in section \ref{Appli}.
We end in the last part with some concluding remarks.

\section{Adaptive Time Splitting Method}\label{AdaptSec}
Let us first set the general mathematical framework of this work.
A class of multi-scale phenomena can be modeled by general 
reaction-diffusion systems of type:
\begin{equation}\label{sys_reac_diff_gen}
\left.
\begin{array}{ll}
\partial _t \mathbf{u} - \partial _{\mathbf{x}} \left( \mathbf{D}(\mathbf{u})
 \partial _{\mathbf{x}} \mathbf{u} \right)
=\mathbf{f}\left(\mathbf{u}\right),&
\quad
\mathbf{x}\in \mathbb{R}^d,
\ t>0,\\[1ex]
\mathbf{u}(0,\mathbf{x})=\mathbf{u}_0(\mathbf{x}),&
\quad \mathbf{x}\in \mathbb{R}^d,
\ t=0,
\end{array}\right\}
\end{equation}
where $\mathbf{f}:\mathbb{R}^{m}\to\mathbb{R}^{m}$ and 
$\mathbf{u}:\mathbb{R}\times \mathbb{R} ^d \to\mathbb{R}^{m}$,
with a tensor of order $d\times d\times m$ as
diffusion matrix $\mathbf{D}(\mathbf{u})$.

In the following we will focus on the simplified  
case of linear diagonal diffusion, 
for which the elements of the diffusion matrix are 
written as 
$D_{i_1i_2i_3}(\mathbf{u})=D_{i_3} \delta_{i_1i_2}$
for some positive indices $i_1$, $i_2$, $i_3$, so that
the diffusion operator reduces to the heat operator
with
some scalar diffusion coefficient $D_{i_3}$ for component
$u_{i_3}$ of $\mathbf{u}$.
A scalar one-dimensional model is considered
in order to simplify the presentation,
taking into account that extension into higher dimensions of $\mathbf{x}$ or
$\mathbf{u}$ is straightforward:
\begin{equation}\label{sys_rea_dif_u}
\left.
\begin{array}{ll}
\partial_t u-\partial^2_{x} u =f(u),& \quad x\in \mathbb{R},\ t>0,\\[1ex]
u(0,x)=u_0(x),& \quad x\in \mathbb{R}, \ t=0,
\end{array}\right\}
\end{equation}
where $f$ and $u_0$ are smooth functions.
We denote by  $T^t u_0$ the solution of \eqref{sys_rea_dif_u}.

Introducing standard decoupling of the diffusion and reaction parts of (\ref{sys_rea_dif_u}),
we denote by $X^tu_0$ the solution of the diffusion equation: 
\begin{equation}\label{split_dif_u}
\begin{array}{ll}
\partial_t u_D- \partial^2_{x} u_D=0,& \quad x\in \R,\ t>0,
\end{array}
\end{equation}
with initial data $u_D(0,\cdot)=u_0(\cdot)$ after some time $t$;
and by $Y^tu_0$, the solution of the reaction part where 
spatial coordinate $x$ can be considered 
as a parameter: 
\begin{equation}\label{split_rea_u}
\begin{array}{ll}
\partial_t u_R=f(u_R),&\quad x\in \R,t>0,
\end{array}
\end{equation}
with $u_R(0,\cdot)=u_0(\cdot)$.

The two Lie approximation formulae of the solution of system (\ref{sys_rea_dif_u}) 
are then defined by
\begin{equation}\label{lie_u}
L^t_{1}u_0=X^tY^{t}u_0,
\quad
L^t_{2}u_0=Y^{t}X^tu_0,
\end{equation}
whereas the two Strang approximation formulae 
\cite{Strang63,Strang68} are given by
\begin{equation}\label{strang_u}
S^t_{1}u_0=X^{t/2}Y^{t}X^{t/2}u_0,
\quad
S^t_{2}u_0=Y^{t/2}X^tY^{t/2}u_0.
\end{equation}

It is well known that Lie formulae (\ref{lie_u}) (resp. Strang formulae (\ref{strang_u})) 
are an approximation of order $1$ (resp. $2$) of the exact solution of (\ref{sys_rea_dif_u}).
Higher order splitting schemes are also possible. 
Nevertheless, 
the order conditions
for such composition methods state
that either negative time substeps or complex coefficients  or non convex combinations are necessary \cite{Hairer02}. 
The formers imply usually important 
stability restrictions and more sophisticated numerical implementations. 
In the particular case of negative time steps, 
they are completely undesirable for PDEs that are ill-posed for
negative time progression.

An adaptive time stepping strategy is based on a local error estimate
which can be obtained by using two schemes of different
order, in this case 
$S^t_{1}$ or $S^t_{2}$,
locally of order $3$, 
and $L^t_{1}$ or $L^t_{2}$,
locally of order $2$.
For instance,
the Embedded Split-Step Formulae given in \cite{Thal} consider 
 $S^t_{1}$ and  $L^t_{2}$ or $S^t_{2}$ and $L^t_{1}$, 
noticing that
 \begin{equation*}
L^t_{1}u_0=X^{t}Y^{t/2}Y^{t/2}u_0,
\end{equation*}
where $Y^{t/2}u_0$ is also used to compute $S^t_{2}u_0$.
Nevertheless,
in the context of multi-scale phenomena,
order reductions may appear due to short-life transients 
associated with the fastest variables
when one considers splitting time steps
larger than the fastest scales.
It has been proved in \cite{Descombes04} 
that better performances are expected while ending
the splitting scheme by 
the part involving the fastest time scales of the phenomenon.
In particular, in the case of linear diagonal diffusion problems, 
no order loss is expected for the $L_2^{t}$ and 
$S_2^{t}$ schemes
when fast scales are present in the reactive term.
Therefore,
the embedding procedure must be carefully conceived
taking into consideration these theoretical studies.

We introduce a shifted Strang formula 
\begin{equation}\label{strangdec}
S^t_{2,\varepsilon}u_0=Y^{(1/2-\varepsilon)t}X^tY^{(1/2+\varepsilon)t}u_0,
\end{equation}
locally of order $2$, due to the lack of symmetry, for $\varepsilon$ in $ [-1/2,0)\cup(0,1/2]$.
In this way, a local error estimate is computed based on
two solutions for which orders are guaranteed and a potential loss of order is simultaneous,
following 
\begin{equation}
\left(
\begin{array}{c}
 S^{\Delta t}_{2}u_0\\[1ex]
S^{\Delta t}_{2,\varepsilon}u_0
\end{array}
\right)
=
\left(
\begin{array}{c}
 Y^{\Delta t/2}X^{\Delta t}
Y^{\Delta t/2}
u_0\\[1ex]
Y^{(1/2-\varepsilon){\Delta t}}X^{\Delta t}
Y^{(1/2+\varepsilon)\Delta t}
%Y^{\varepsilon{\Delta t}}
%Y^{\Delta t/2}
u_0
\end{array}
\right),
\end{equation}
for some splitting time step $\Delta t >0$.
Embedding is accomplished as long as $\varepsilon$ is 
different from  $-1/2$,
that is $S^{\Delta t}_{2,\varepsilon}u_0$ different from $L^t_{2}u_0$.
On the other hand,
if  $\varepsilon$ is equal to  $1/2$,
$S^{\Delta t}_{2,\varepsilon}u_0$ is defined as $L^t_{1}u_0$,
which it is not suitable for stiff configurations as it was previously discussed
\cite{Descombes04}.
Therefore, 
$\varepsilon$ should be contained in  $(-1/2,0)\cup(0,1/2)$.
Shifted $S^{\Delta t}_{1,\varepsilon}u_0$
is defined in a similar way and 
depending on the multi-scale character of the problem, it
might be the appropriate choice
along with $S^{\Delta t}_{1}u_0$.

Taking into account that
\begin{eqnarray}\label{differ}
S^{\Delta t}_{2}u_0-
S^{\Delta t}_{2,\varepsilon}u_0
&=& 
S^{\Delta t}_{2}u_0-T^{\Delta t}u_0+
T^{\Delta t}u_0
-
S^{\Delta t}_{2,\varepsilon}u_0,
\nonumber \\ 
 &=&
 \mathcal{O}(\Delta t^3) +  \mathcal{O}(\Delta t^2) \approx \mathcal{O}(\Delta t^2),
  \end{eqnarray}
for a given accuracy tolerance $\eta$,
\begin{equation}\label{est_split_er_tol}
 \big\|S^{\Delta t}_{2}u_0-
S^{\Delta t}_{2,\varepsilon}u_0\big\| 
< \eta
\end{equation}
must be verified
in order to accept current computation with $\Delta t$,
while new time step is calculated by
\begin{equation}\label{delta_split_ef}
\Delta t^{\rm new} = \upsilon \, \Delta t  
\sqrt{\frac{\eta}
{\big\|S^{\Delta t}_{2}u_0-
S^{\Delta t}_{2,\varepsilon}u_0\big\|}},
\end{equation}
with security factor $0<\upsilon \leq 1$ close to one. 
This comes from a classical adaptive time stepping procedure
for stiff ODEs solution,
for which 
more sophisticated formulae than (\ref{delta_split_ef}) can be also considered, 
see \cite{MR1227985} for example.

The error control of these adaptive methods is fully guaranteed as long as the
orders of both, the main and the embedded integration methods, remains
valid. This is the case for small enough time steps for which asymptotic
theoretical estimates hold, 
but remains an open problem for larger time steps for which 
the validity of the formers is assumed.
This is a key point in this work, because we propose not only
a new splitting strategy with adaptive time steps as described in this section, 
but we aim also
at applications for which splitting time steps may go beyond
the fastest scales associated with each subproblem
in order to obtain important computational savings.
Therefore, a technique that guarantees consistently error control for
all possible separation scales must be pursued,
but first of all, a detailed numerical analysis of the method must be performed.
This is the goal of the following part.

\section{Numerical Analysis of the Method}\label{analyse}
In this part,
we develop the numerical analysis of the proposed method.
It is mainly based on the theoretical study of the introduced
shifted Strang formula (\ref{strangdec})
and the domain of validity of the local error estimates.
However, first of all, we introduce the Lie formalism which will
be used as mathematical tool of analysis.

\subsection{The Lie operator formalism}	
We introduce 
the Lie operator formalism in order to 
generalize the exponential of a linear operator in the context of nonlinear
operators. Let $X$ be a Banach space, $T_0 > 0$ and
$F$, an unbounded nonlinear operator from $D(F) \subset  X$ to $X$, we consider the general autonomous equation:
\begin{equation}\label{thesystem}
\left.
\begin{array}{ll}
{\displaystyle
u'(t) = F(u(t)),}
& \quad 0 < t <T_0,
\\[1ex]
u(0) = u_0,
& \quad t=0.
\end{array}
\right\}
\end{equation}
The exact solution of this evolutionary equation is (formally) given by
\begin{equation}\label{exact}
u(t)=T^t u_0, \quad 0 \leq t \leq T_0
\end{equation}
where $T^t$ is the semiflow associated with \eqref{thesystem};
in particular we can set $F(u)=\partial^2_{x} u + f(u)$ as in (\ref{sys_rea_dif_u}).
The Lie operator $D_F$ associated with $F$ is then a linear operator acting 
on the space of operators defined in $X$ \cite{Hairer02,DesThal}.
More precisely, 
for any
unbounded nonlinear operator $G$ 
from $D(G) \subset X$ to $X$ with Fr\'echet derivative $G'$,
$D_F$ maps $G$
into a new operator $D_F G$,
such that for any $v$ in $X$:
\begin{equation}\label{LieOp}
(D_F G)(v) = G'(v) F(v).
\end{equation}
Hence, by induction on $n$ with solution $u$
of \eqref{thesystem}, we obtain
\begin{equation*}
\frac{\partial^n}{\partial t^n}
G( u(t)) = (D^n_F G)(u(t)),
\end{equation*}
and a formal Taylor expansion yields
\begin{equation}\label{Taylor}
G(u(t))=
\sum_{n=0}^{+ \infty}
\frac{t^n}{n !}
\left (
\frac{\partial^n}{\partial t^n} G( u(t))\right ) \bigg \vert_{t=0}
=
\left (
\sum_{n=0}^{+ \infty}
\frac{t^n}{n !}
D^n_F G
\right )
u_0
=
\left (
 e^{t D_F} G 
\right )
u_0.
\end{equation}
If we now assume that $G$ is the identity operator $\textrm{Id}$, we obtain
\begin{equation*}
u(t)= T^t u_0 = 
\left (
 e^{t D_F}  \textrm{Id}
\right )
u_0.
\end{equation*}
Therefore,
the Lie operator is indeed a way to write the solution of a nonlinear equation
in terms of a linear but differential operator.
Following \eqref{Taylor},
an important result obtained by Gr\"obner in 1960 \cite{Grobner1960}, 
considers
the composition of two semiflows $T_1^t$ and $T_2^s$ associated with $F_1$ and $F_2$
for any $v$ in  $X$:
\begin{equation*}
T_1^t  T_2^s v = \left (
 e^{s D_{F_2}} T_1^t 
\right )
v =
 \left (
 e^{s D_{F_2}}  e^{t D_{F_1}}  \textrm{Id}
\right )
v. 
\end{equation*}

\subsection{Error analysis}
In this paragraph,
we conduct the error analysis of the approximation
of $T^t$ by $S^t_{2,\varepsilon}$ 
in a linear framework.
Then, we
extend these results to a general nonlinear configuration
given by problem \eqref{sys_rea_dif_u},
using the Lie operator formalism.
General estimates for
the approximation of $T^t$ by $S_{2}^t$
are also drawn.
We end in the last part 
with a mathematical study that shows the
domain of application of the method described 
in $\S~$\ref{AdaptSec},
for which
an effective error control
is guaranteed
within an accuracy tolerance.
To simplify
the notations
in what follows, 
we will denote $S_{2}^t$ by $S^t$ and
$S^t_{2,\varepsilon}$ by $S^t_{\varepsilon}$.

Assume that
$A$ and $B$ are linear bounded operators and define
\begin{equation*}
S^t_{\varepsilon} u_0 = e^{(1/2-\varepsilon)tA}e^{tB}e^{(1/2+\varepsilon)tA} u_0
\end{equation*}
as an approximation of $e^{t(A+B)}$.
The following theorem gives the expansion in powers of $t$
of the difference between $e^{t(A+B)}$ and $S^t_{\varepsilon}$.
We recall the definition of the brackets between $A$ and $B$:
$
[A,B] = AB - BA$.
\begin{theorem}\label{linear}
Assume that $A$ and $B$ are linear bounded operators,
for $t$ and $\varepsilon$ small enough, the following asymptotic holds
\begin{equation*}
e^{t(A+B)} u_0  - S^t_{\varepsilon}  u_0 = -  \varepsilon t^2 [A,B] u_0 +  \frac{t^3}{24}
(
\big[A,[A,B]\big]
 + 2 \big[B,[A,B]\big]) u_0
 +\mathcal{O}(\varepsilon t^3)
 +\mathcal{O}(t^4).
  \end{equation*}
\end{theorem}
\begin{proof}
Proof is straightforward by using the Taylor formula with integral remainder
for a linear bounded operator  $A$:
\begin{equation*}
e^{tA}  = \text{Id} +t A + \frac{t^2 A^2}{2} +
 \frac{t^3 A^3}{6} 
+ \int_0^t \frac{(t-s)^3}{6} A^4  e^{sA} \, {\rm d} s.
\end{equation*}
\end{proof}
We extend now the previous theorem to our nonlinear framework given by
(\ref{sys_rea_dif_u}).
In order to do this,
we introduce the spaces
$C^{\infty}(\R)$ of functions of class $C^{\infty}$ on $\R$,
and
$C^{\infty}_{b}(\R)$ of functions of class $C^{\infty}$ on $\R$
and bounded over  $\R$.
%equipped with the norm $\| \cdot \|_{\infty}$.
We consider also  the Schwartz space $\mathcal{S}(\R)$ 
defined by 
\begin{equation*}
 \mathcal{S} \left(\mathbb{R}\right) = \{ g \in C^\infty(\mathbb{R})
  \mid 
  \sup_{v\in\mathbb{R}}|v^{\alpha_1} \partial_v^{\alpha_2} g(v)|
   < \infty\quad\,\textrm{for\,all\,integers}\,\alpha_1, \alpha_2 \};
  \end{equation*}
and we define the space $\mathcal{S}_1(\R)$, made out of functions
$v$ belonging to $C^{\infty}_{b}(\R)$ such that $v'$ belongs to $\mathcal{S}(\R)$.
Let us consider now equation (\ref{sys_rea_dif_u}) and give
the expansion in powers of $t$
of the difference between $T^t$ and $S^t_{\varepsilon}$, given by (\ref{strangdec}).
\begin{theorem}\label{theo1}
Assume that $u_0$ belongs to $\mathcal{S}_1 \left(\mathbb{R}\right) $
and that $f$ belongs to $C^{\infty}(\R)$.
For $t$ and $\varepsilon$ small enough, the following asymptotic holds
\begin{eqnarray}\label{estim1}
T^tu_0 - S^t_{\varepsilon}  u_0 &=& - \varepsilon t^2
f''(u_0) \left ( \frac{\partial u_0}{\partial x} \right )^2
\nonumber \\ 
&&+ \frac{t^3}{24}
  (  f'(u_0)f''(u_0)
+ f(u_0) f^{(3)}(u_0) ) 
 \left ( \frac{\partial u_0}{\partial x} \right )^2\nonumber \\ 
 &&- \frac{ t^3}{12}
f^{(4)}(u_0)  \left ( \frac{\partial u_0}{\partial x} \right )^4
- \frac{ t^3}{3}
f^{(3)}(u_0)  \left ( \frac{\partial u_0}{\partial x} \right )^2
 \frac{\partial^2 u_0}{\partial x^2}\nonumber \\
&&- \frac{ t^3}{6}  f''(u_0)  \left ( \frac{\partial^2 u_0}{\partial x^2} \right )^2
+\mathcal{O}(\varepsilon t^3)+\mathcal{O}(t^4).
  \end{eqnarray}
\end{theorem}
\begin{proof}
We introduce the two Lie operators 
$D_{\Delta}$ and $D_{f}$ associated with
$\partial^2_x$ and $f$ and write
\begin{equation*}
T^tu_0 - S^t_{\varepsilon}  u_0  =
\left ( e^{t (D_{\Delta}+D_{f})} \rm{Id} \right ) u_0- \left ( 
e^{(1/2+\varepsilon)tD_{f}}
e^{tD_{\Delta}}
e^{(1/2-\varepsilon)tD_{f}} 
 \rm{Id} \right )  u_0.
\end{equation*}
With Theorem \ref{linear} we can deduce that
\begin{eqnarray}\label{est_gen}
T^tu_0 - S^t_{\varepsilon}  u_0 &=&-  \varepsilon t^2
\left (
[D_{f},
D_{\Delta}] 
 \rm{Id} \right ) u_0+  \frac{ t^3}{24}
 \left (
 \big[D_{f}, [D_{f},D_{\Delta}] \big]  \rm{Id} \right ) u_0
 \nonumber\\ 
 &&+
  \frac{t^3}{12}
  \left (
 \big[D_{\Delta},[D_{f},D_{\Delta}]\big]  \rm{Id} \right ) u_0
 +\mathcal{O}(\varepsilon t^3)+\mathcal{O}(t^4).
  \end{eqnarray}
  We are not interested in giving the exact form of the terms $\mathcal{O}(\varepsilon t^3)$
  and $\mathcal{O}(t^4)$, but these terms
  can be computed following the same technique developed in \cite{DesThal}.
  For the term in $\mathcal{O}(t^2)$, we have by definition and with \eqref{LieOp},
  \begin{eqnarray*}
\left (
[D_{f},D_{\Delta}] \rm{Id} 
\right ) u_0 &=&
\left(
D_{f} ( D_{\Delta}  \mathrm{Id} )
- D_{\Delta} ( D_f  \rm{Id} )
\right)
 u_0
,\\ \nonumber
& = &
(D_{\Delta}  \mathrm{Id})'(u_0) f(u_0)
-
( D_{f}  \mathrm{Id})'(u_0) \frac{\partial^2 u_0 }{\partial x^2},\\ \nonumber
&= & \frac{\partial^2 }{\partial x^2} \left (f(u_0) \right )
-f'(u_0) \frac{\partial^2 u_0 }{\partial x^2}.
\end{eqnarray*}
The last term is by definition the Lie bracket between
$\partial^2_x$ and $f$,
a simple computation shows that 
\begin{eqnarray*}\label{commutLie}
\frac{\partial^2 f(u_0) }{\partial x^2}  - f'(u_0) \frac{\partial^2 u_0}{\partial x^2}
&=&  
f''(u_0) \left ( \frac{\partial u_0}{\partial x} \right )^2 
+f'(u_0) \frac{\partial^2 u_0}{\partial x^2}-f'(u_0) \frac{\partial^2 u_0}{\partial x^2},\\ \nonumber
&=& f''(u_0) \left ( \frac{\partial u_0}{\partial x} \right )^2.
\end{eqnarray*}
Furthermore,
\begin{equation*}
  \left (
 \big[D_{f}, [D_{f},D_{\Delta}]\big]  \rm{Id} \right ) (u_0) =
(  f'(u_0) f''(u_0)
+ f(u_0) f^{(3)}(u_0) ) \left ( \frac{\partial u_0}{\partial x} \right )^2
\end{equation*}
and
\begin{eqnarray*}
 \left (
 \big[D_{\Delta}, [D_{f},D_{\Delta}]\big]  \rm{Id} \right ) u_0
&=&
- f^{(4)}(u_0)  \left ( \frac{\partial u_0}{\partial x} \right )^4
-4 f^{(3)}(u_0)  \left ( \frac{\partial u_0}{\partial x} \right )^2
 \frac{\partial^2 u_0}{\partial x^2} \\
&&- 2  f''(u_0)  \left ( \frac{\partial^2 u_0}{\partial x^2} \right )^2.
\end{eqnarray*}
All the terms are now computed and 
this concludes the proof of Theorem \ref{theo1}.
\end{proof}
For $\varepsilon=0$, the next corollary follows directly.
\begin{corollary}
Assume that $u_0$ belongs to $\mathcal{S}_1 \left(\mathbb{R}\right) $
and that $f$ belongs to $C^{\infty}(\R)$.
For $t$ small enough, the following asymptotic holds
\begin{eqnarray}\label{estim2}
T^tu_0 - S^t u_0 &=&
 \frac{t^3}{24}
  (  f'(u_0)f''(u_0)
+ f(u_0) f^{(3)}(u_0) ) 
 \left ( \frac{\partial u_0}{\partial x} \right )^2\nonumber \\ 
 &&- \frac{ t^3}{12}
f^{(4)}(u_0)  \left ( \frac{\partial u_0}{\partial x} \right )^4
- \frac{ t^3}{3}
f^{(3)}(u_0)  \left ( \frac{\partial u_0}{\partial x} \right )^2
 \frac{\partial^2 u_0}{\partial x^2}\nonumber \\
&&- \frac{ t^3}{6}  f''(u_0)  \left ( \frac{\partial^2 u_0}{\partial x^2} \right )^2
+\mathcal{O}(t^4).  \end{eqnarray}
\end{corollary}
From (\ref{estim1}) and (\ref{estim2}), we can see that
\begin{equation}\label{s_seps}
S^t u_0  - S^t_{\varepsilon}  u_0 = \varepsilon t^2
f''(u_0) \left ( \frac{\partial u_0}{\partial x} \right )^2 + \mathcal{O}(\varepsilon t^3),
\end{equation}
and thus,
\begin{equation}\label{somme_err}
T^tu_0 - S^t_{\varepsilon}  u_0 = \underbrace{T^tu_0 - S^t u_0}_{\mathcal{O}(t^3)} 
+ \underbrace{S^t u_0  - S^t_{\varepsilon}  u_0}_{\mathcal{O}(\varepsilon t^2)}.
\end{equation}
Therefore,
we are sure that 
the real local error of the method, $T^tu_0 - S^t u_0$, will be bounded by
the local error estimate, $err=S^t u_0  - S^t_{\varepsilon}  u_0$,
when for a given $\varepsilon$,
\begin{equation}\label{order_cond}
 T^tu_0 - S^t_{\varepsilon}  u_0 \approx \mathcal{O}(t^2)
\end{equation}
is verified into (\ref{somme_err});
that is, 
when the embedded method is really of lower order
as it was assumed in (\ref{differ}).
This will be always verified for small enough time steps $t$,
for which $T^tu_0 - S^t u_0 \approx  \mathcal{O}(t^3) < err \approx  \mathcal{O}(\varepsilon t^2)$
is guaranteed.
Nevertheless, 
for larger time steps,
$err$ will fail to properly predict $T^tu_0 - S^t u_0$
since we will eventually have 
$T^tu_0 - S^t u_0 \approx  \mathcal{O}(t^3) > err \approx  \mathcal{O}(\varepsilon t^2)$.
When this happens, (\ref{order_cond})
is no longer true and the previous estimates show that we will rather have
$T^tu_0 - S^t_{\varepsilon}  u_0 \approx \mathcal{O}(t^3)$,
and assumption (\ref{differ}) will no longer hold.

In order to overcome this difficulty,
we must therefore estimate a critical time step
$t^\star >0$ 
such that
for all $t$ in $[0,t^*]$,
(\ref{order_cond}) is guaranteed for a given $\varepsilon$.
This will imply that Strang local error,
$T^tu_0 - S^t u_0$, will be indeed bounded by the local error
estimate, $err$,
and that 
an effective error control will be achieved for $err$
smaller than a given accuracy tolerance $\eta$.
Finally, a suitable choice of $\varepsilon$
can be also made since $t^\star$ is related to
$\varepsilon$ following (\ref{somme_err}).

A natural strategy to predict this critical 
$t^\star$ will rely on the
previous theoretical estimates
and on a more precise knowledge of the structure of the
solutions of the PDEs;
this is for instance illustrated in the next part in the context of
traveling wave solutions.

\section{Application to Reaction Traveling Waves}\label{KPP}
In this part,
we will confront the previous theoretical study to a simple reaction
diffusion problem that admits self-similar traveling wave solutions
such as the KPP equation \cite{KPP38}.
The main advantages of considering this kind of problems are that
analytic solutions exist and that the featured stiffness can be tuned
using a space-time scaling.
Therefore, it provides a first numerical validation
of the numerical estimates of the method and 
an evaluation
of its domain of application; and
on the other hand,
a detailed study can be conducted on the impact of the stiffness featured  
by propagating fronts with steep spatial gradients.

In what follows, we recast previous estimates in the context
of these reaction traveling waves, to
then deduce an estimate of the time step $t^\star$
that defines the limit of application of the method
for which local error estimates yield effective error control.
We end with a numerical validation
of the theoretical results in the context of the resolution of KPP model.

\subsection{Numerical estimates}
We are interested in the propagation of self-similar
waves modeled by parabolic PDEs of type: 
\begin{equation}\label{sys_rea_dif_wave}
\left.
\begin{array}{ll}
\partial_t u-D\, \partial^2_{x} u =kf(u),& \quad x\in \mathbb{R},\ t>0,\\[1ex]
u(0,x)=u_0(x),&\quad x\in \mathbb{R}, \ t=0,
\end{array}\right\}
\end{equation}
with solution $u(x,t)=u_0(x - ct)$,
where $c$ is the steady speed of the wavefront,
and $D$ and $k$ stand respectively for
diffusion and reaction coefficients.

Considering Theorem \ref{theo1}
we obtain the following estimate for system (\ref{sys_rea_dif_wave}).
\begin{corollary}
Assume that $u_0$ belongs to $\mathcal{S}_1 \left(\mathbb{R}\right) $
and that $f$ belongs to $C^{\infty}(\R)$.
For $t$ and $\varepsilon$ small enough, the following asymptotic holds
\begin{eqnarray}\label{estim1_kd}
T^tu_0 - S^t_{\varepsilon}  u_0 &=& - \varepsilon kD t^2
f''(u_0) \left ( \frac{\partial u_0}{\partial x} \right )^2
\nonumber \\ 
&&+ \frac{k^2 D t^3}{24}
  (  f'(u_0)f''(u_0)
+ f(u_0) f^{(3)}(u_0) ) 
 \left ( \frac{\partial u_0}{\partial x} \right )^2\nonumber \\ 
 &&- \frac{k D^2 t^3}{12}
f^{(4)}(u_0)  \left ( \frac{\partial u_0}{\partial x} \right )^4
- \frac{k D^2 t^3}{3}
f^{(3)}(u_0)  \left ( \frac{\partial u_0}{\partial x} \right )^2
 \frac{\partial^2 u_0}{\partial x^2}\nonumber \\
&&- \frac{ k D^2 t^3}{6}  f''(u_0)  \left ( \frac{\partial^2 u_0}{\partial x^2} \right )^2
+\mathcal{O}(\varepsilon t^3)+\mathcal{O}(t^4).
  \end{eqnarray}
  \end{corollary}
\begin{proof}
Proof follows directly from demonstration of Theorem \ref{theo1},
using (\ref{est_gen}) and  
considering that
\begin{eqnarray*}
[D_{k f},D_{D \Delta}] & = & kD[D_{f},D_{\Delta}], \\
\big[[D_{k f},D_{D \Delta}],D_{D \Delta}\big] & = & kD^2\big[[D_{f},D_{\Delta}],D_{\Delta}\big], \\
\big[[D_{k f},D_{D \Delta}],D_{k f}\big] & = & k^2D\big[[D_{f},D_{\Delta}],D_{f}\big],
\end{eqnarray*}
where $D_{D \Delta}$ and $D_{k f}$ are the Lie operators associated with
$D \partial^2_x$ and $k f$.
\end{proof}

On the other hand, 
if we now consider system (\ref{sys_rea_dif_wave})
with $k=1$ and $D=1$,
the 
following corollary establishes
$t^* >0$ 
such that
for all $t$ in $[0,t^*]$
(\ref{order_cond}) is guaranteed for a given $\varepsilon$.
\begin{corollary}\label{corostar}
Assume that $u_0$ belongs to $\mathcal{S}_1 \left(\mathbb{R}\right) $
and that $f$ belongs to $C^{\infty}(\R)$.
For a given $\varepsilon$ small enough , define
\begin{equation}\label{Mone}
M_1 = \left \|   f''(u_0)  \left ( \frac{\partial u_0}{\partial x} \right )^2  \right \|_{L^2}
\end{equation}
and
\begin{eqnarray}\label{Mtwo}
M_2
&=
\Biggl \|  &
 \frac{ f'(u_0)f''(u_0)
+ f(u_0) f^{(3)}(u_0) }{24}
 \left ( \frac{\partial u_0}{\partial x} \right )^2
 - \frac{f^{(4)}(u_0)}{12}
 \left ( \frac{\partial u_0}{\partial x} \right )^4
 \nonumber \\
 &&- 
\frac{f^{(3)}(u_0)}{3}
  \left ( \frac{\partial u_0}{\partial x} \right )^2
 \frac{\partial^2 u_0}{\partial x^2}- \frac{f''(u_0)}{6}   \left ( \frac{\partial^2 u_0}{\partial x^2} \right )^2  \Biggr \|_{L^2},
  \end{eqnarray}
  define $ t^\star$ by
\begin{equation}\label{tstar}
t^\star M_2 =  \varepsilon M_1.
\end{equation}
For all $t$ such that 
   $0<t\leq t^\star$ then 
$$
\|  T^{ t} u_0 - S^{ t} _{\varepsilon}  u_0 \|_{L^2} \approx \mathcal{O}(t^2).
$$
\end{corollary}

In a general case,
if
evaluation of the derivatives of $u_0$ and $f$ is feasible,
it is then possible to predict the domain of application of the method,
$[0,t^\star]$, for a given $\varepsilon$
based on the previous result.
In the particular case of traveling wave solutions for (\ref{sys_rea_dif_wave}),
diffusion and reaction coefficients, $D$ and $k$, might be seen as 
scaling coefficients in time and space.
A dimensionless analysis of a traveling wave,
as shown in \cite{Scott94}, can be then conducted
considering a dimensionless time $\tau$ and a dimensionless space $r$ with
 $$\tau = kt \quad \text{and} \quad r=(k/D)^{1/2}x.$$
This analysis allows to find
a steady velocity of the wavefront, 
\begin{equation}\label{wave_speed}
\displaystyle c=x_t \propto (Dk)^{1/2},
\end{equation}
whereas
the sharpness of the wave profile is measured
by 
\begin{equation}\label{wave_grad}
\displaystyle \left.{u}_{x}\right|_{\max} \propto (k/D)^{1/2}.
\end{equation}
Therefore, 
condition $Dk=1$ implies constant velocity for all $k=1/D$
but 
greater $k$ (or 
smaller $D$) implies higher spatial gradients, and thus, stiffer configurations.

This study gives complementary information on the
solution of (\ref{sys_rea_dif_wave})
and in particular, when condition 
$Dk=1$ is satisfied,
it allows to deduce from Corollary \ref{corostar}:
 \begin{equation}\label{tstar2}
 kt^\star M_2 =  \varepsilon M_1,
 \end{equation}
with $M_1$ and $M_2$ given by (\ref{Mone}) and (\ref{Mtwo}).
Therefore, stiffer configurations given by the presence of steeper spatial gradients
will restrain the application domain of the method, according to (\ref{tstar2}).
Nevertheless, for larger gradients,
smaller time steps are also required
for a given
level of accuracy
and hence, we can expect a simultaneous reduction of both critical and accurate
splitting time steps.

\subsection{Numerical illustration: KPP equation}\label{KPP_res}
Let us recall the Kolmogorov-Petrovskii-Piskunov
model. 
In their original paper \cite{KPP38}, these authors 
introduced a model describing the propagation of a virus and 
the first rigorous analysis of a stable traveling wave solution of a
nonlinear reaction-diffusion 
equation \cite{Scott94}.
The equation is the following:
\begin{equation}\label{eq:kpp}
\partial_t u - D\,\partial^2_x u = k\,u^2(1-u),
\end{equation}
with homogeneous Neumann boundary conditions.
We consider a
1D discretization with $5001$ points on a $[-70,70]$ region
for which we have negligible spatial discretization errors with respect to
the ones coming from the numerical time integration.

The description of the dimensionless model and the structure of the exact solution 
can be found in \cite{Scott94}
where the dimensionless analysis shows that 
in the case of  $D=1$ and $k=1$, 
the velocity of the self-similar traveling wave is
$c=1/\sqrt{2}$ and the maximal gradient value reaches $1/\sqrt{32}$. 
The key point of this illustration is  
that the velocity of the traveling wave is proportional to 
$(k\,D)^{1/2}$, whereas the maximal gradient is proportional to 
$(k/D)^{1/2}$. 
Hence, we consider the case $kD=1$ for which one may obtain steeper
gradients for the same speed of propagation.

Throughout all this paper, exact solution
$T^t u_0$ will be approximated by 
the resolution of the 
coupled reaction-diffusion problem
performed by the Radau5 method \cite{Hairer96} with fine tolerances,
$\eta_{Radau5}=10^{-10}$.
This solution will be referred as the reference or {\it quasi-exact} solution.
Strang approximations 
$S^t u_0$ and $S^t_\varepsilon u_0$
will be computed with a splitting technique
recently introduced \cite{article_mr,article_avc},
which considers
Radau5 \cite{Hairer96}
to solve locally point by point
the reaction term;
and the ROCK4 method \cite{Abdulle02}
for the diffusion problem.
Radau5 \cite{Hairer96}
is a fifth order implicit Runge-Kutta method 
exhibiting \textit{A}- and 
\textit{L}-stability
properties to efficiently solve stiff systems of ODEs,
whereas ROCK4 \cite{Abdulle02} is formally
a fourth order
\textit {stabilized} explicit Runge-Kutta method
with extended
stability domain along the negative real axis,
well suited to numerically treat mildly stiff elliptic operators.
Both methods implement adaptive time stepping techniques
to guarantee computations within a prescribed
accuracy tolerance.
In order to properly discriminate the previously estimated splitting errors 
from those coming from temporal integration of the substeps, we consider
also fine tolerances, 
$\eta_{\rm Radau5}=\eta_{\rm ROCK4}=10^{-10}$.
\begin{figure}[!htb]
 \begin{center}
 \includegraphics[width=0.45\hsize]{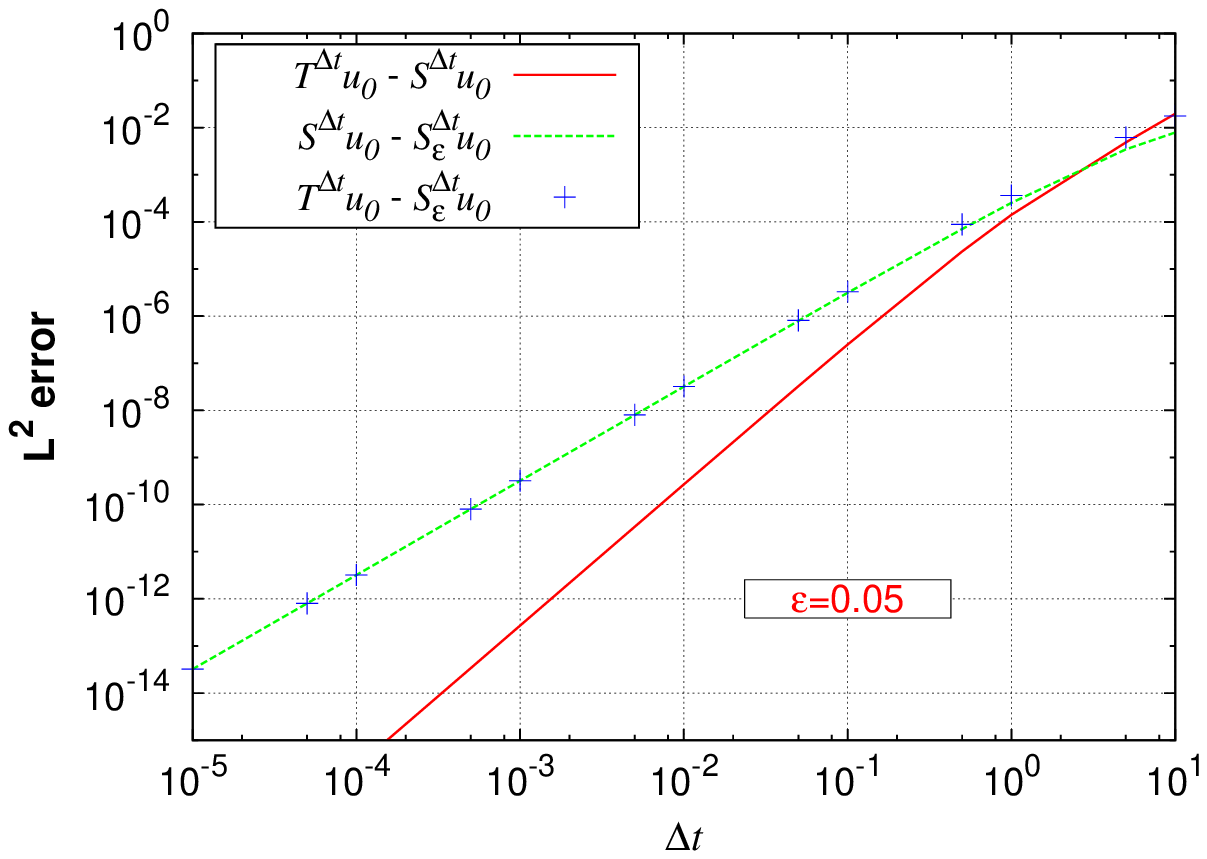}
 \includegraphics[width=0.45\hsize]{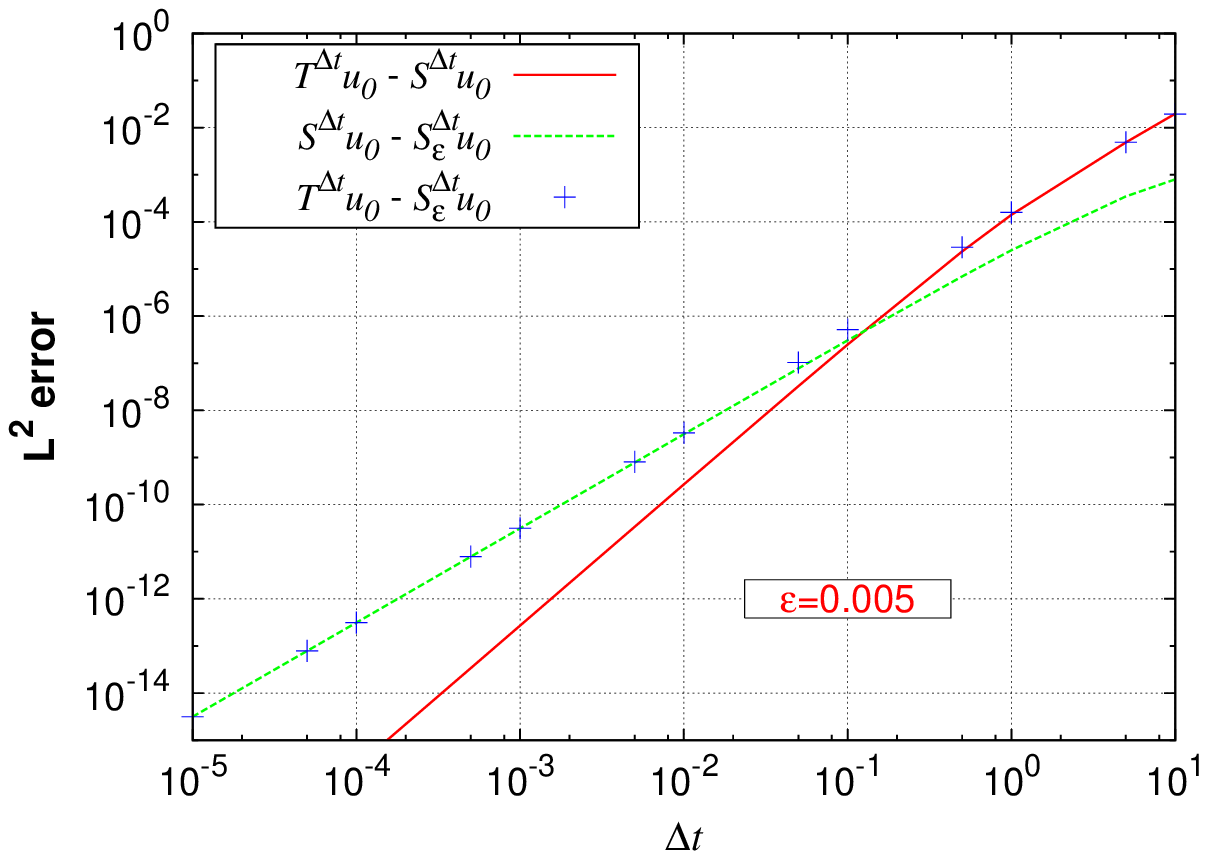}  
 \includegraphics[width=0.45\hsize]{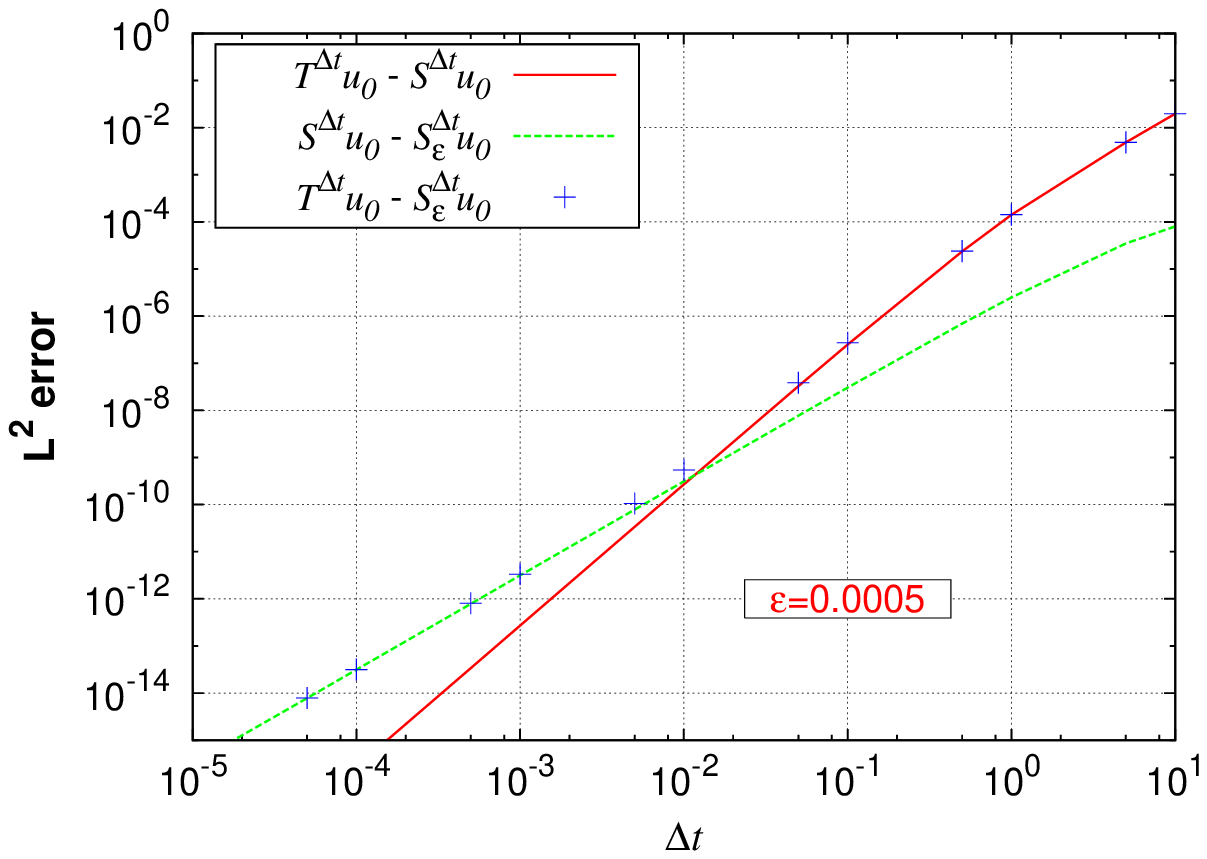}
\includegraphics[width=0.45\hsize]{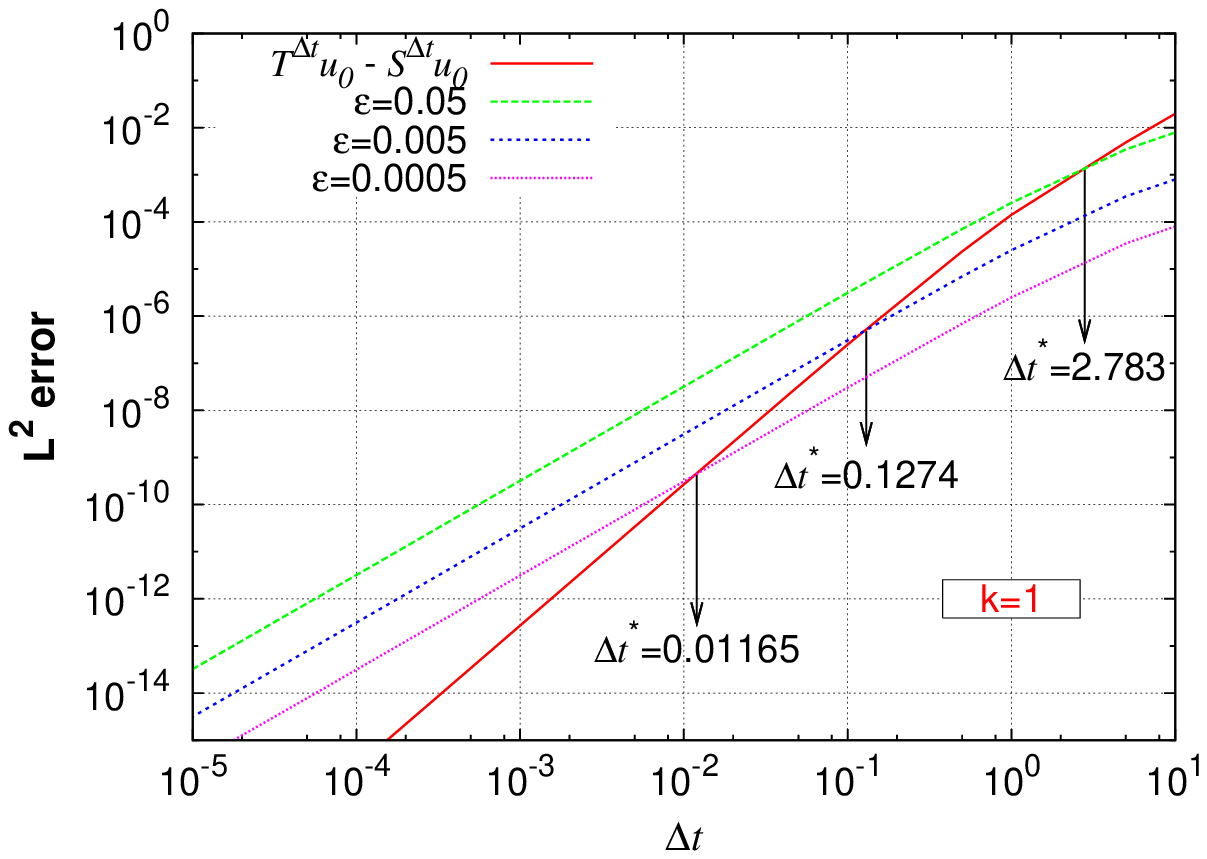} 
   \caption{KPP equation with $k=1$.
            Local $L^2$ errors for several splitting time steps $\Delta t$ and
$\varepsilon=0.05$ (top left), $0.005$ (top right) and $0.0005$ (bottom left).
 Bottom right: critical splitting time steps $\Delta t^\star$
obtained when $\|T^{\Delta t}u_0-S^{\Delta t}u_0\|_{L^2} \approx \|S^{\Delta t}u_0-S^{\Delta t}_{\varepsilon}u_0\|_{L^2}$
in the numerical tests.
 \label{kpp_error}}
 \end{center}
\end{figure}

Figures \ref{kpp_error} and \ref{kppk10_error}
show $L^2$ errors between $T^t u_0$, $S^t u_0$ and $S^t_\varepsilon u_0$
solutions for $k=1$, $k=10$ and $k=100$ respectively, and several $\varepsilon$.
Notice that 
estimates (\ref{estim1}), (\ref{estim2}) and (\ref{s_seps})
for all three errors in (\ref{somme_err})
are verified and in particular,
for $\Delta t$ larger than critical $\Delta t^\star$, the estimated error 
$err=\|S^{\Delta t}u_0-S^{\Delta t}_{\varepsilon}u_0\|_{L^2}$
is no longer predicting the real local error given by $T^tu_0 - S^t u_0$.
\begin{figure}[!htp]
 \begin{center}
 \includegraphics[width=0.45\hsize]{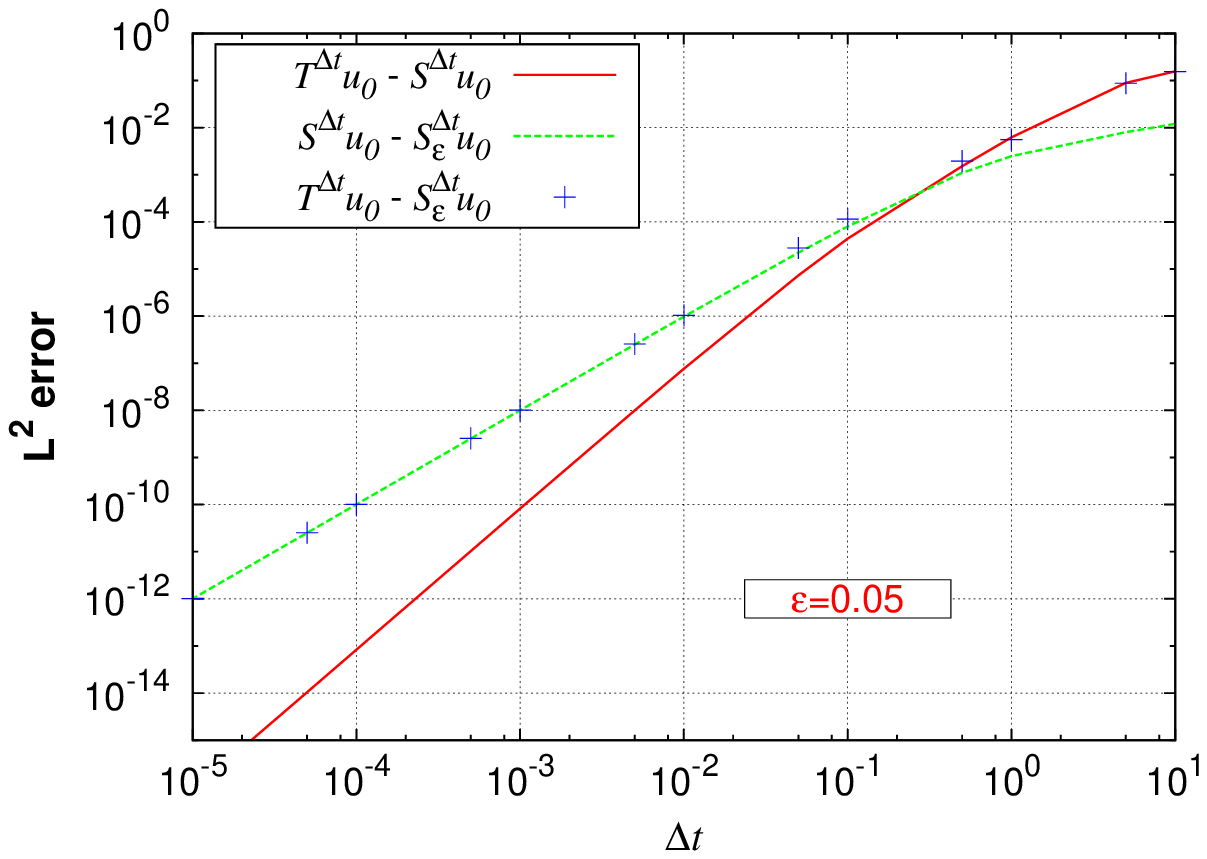}
 \includegraphics[width=0.45\hsize]{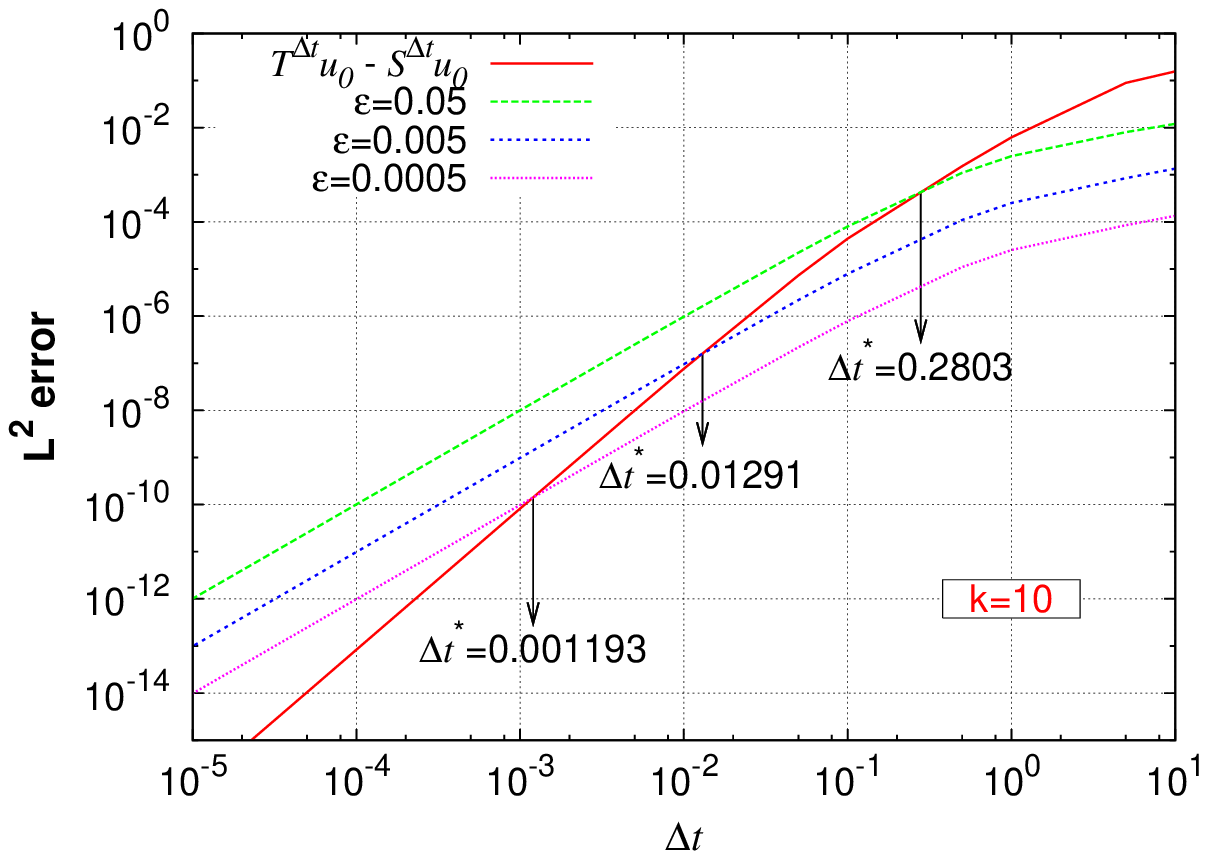}  
 \includegraphics[width=0.45\hsize]{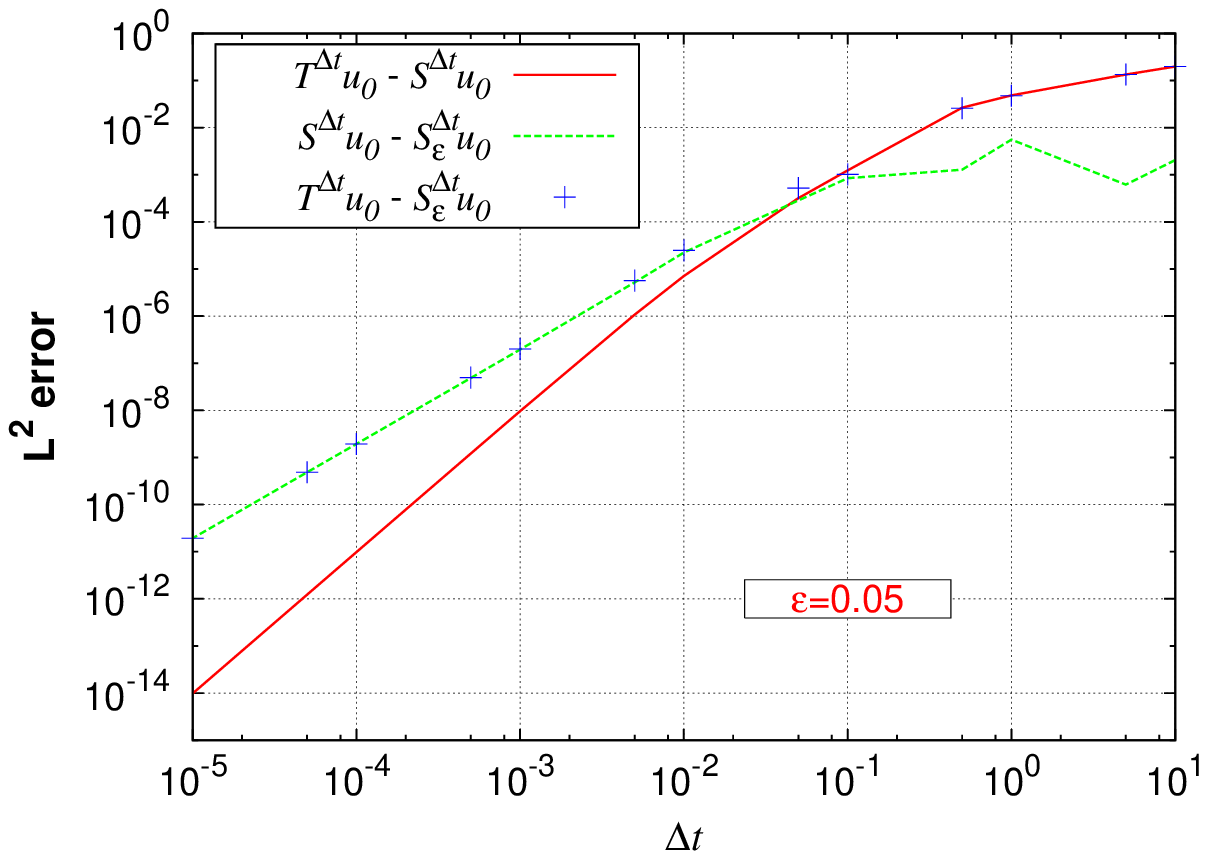}
 \includegraphics[width=0.45\hsize]{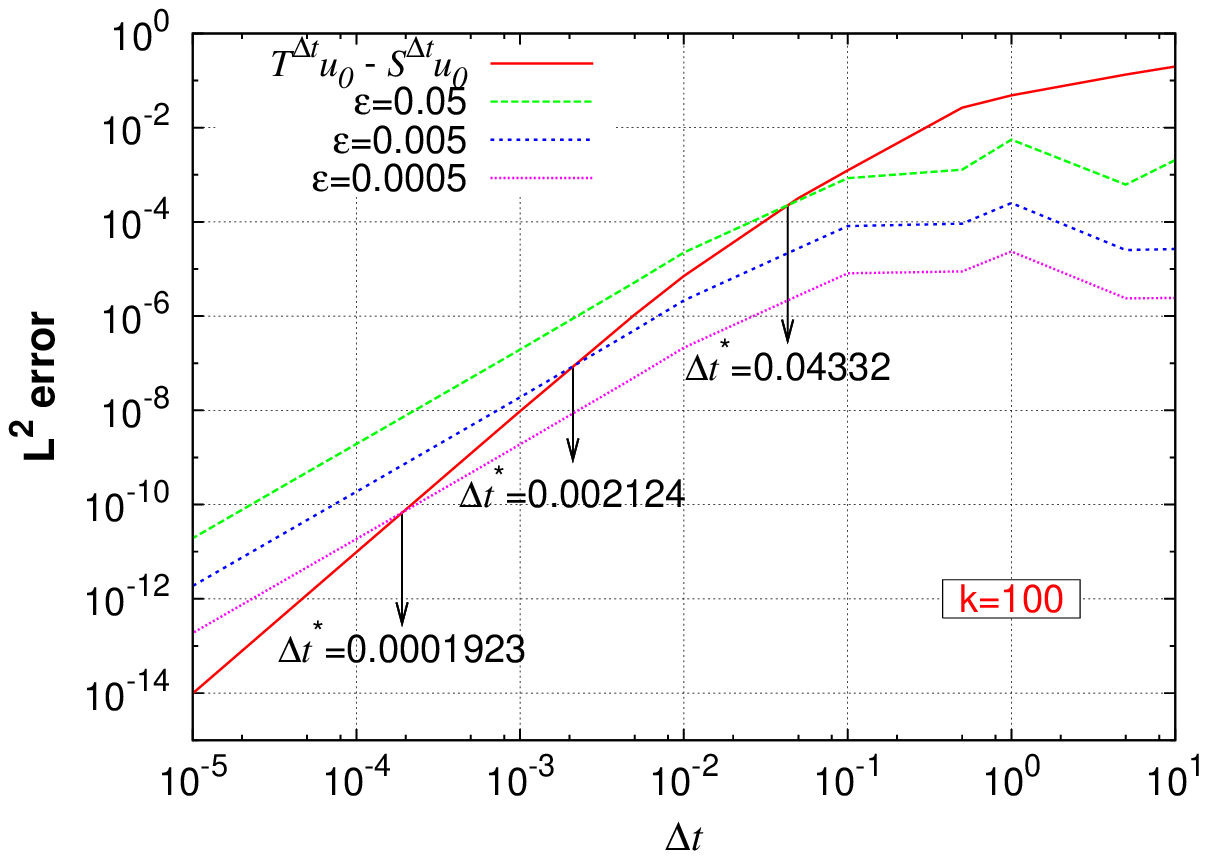} 
   \caption{KPP equation with $k=10$ (top) and $k=100$ (bottom).
            Local $L^2$ errors for several splitting time steps $\Delta t$ and
$\varepsilon=0.05$ (left).
Right: critical splitting time steps $\Delta t^\star$
obtained when $\|T^{\Delta t}u_0-S^{\Delta t}u_0\|_{L^2} \approx \|S^{\Delta t}u_0-S^{\Delta t}_{\varepsilon}u_0\|_{L^2}$
in the numerical tests.
 \label{kppk10_error}}
 \end{center}
\end{figure}

With these results, 
we can also compare real $\Delta t^\star$,
obtained when
$\|T^{\Delta t}u_0-S^{\Delta t}u_0\|_{L^2} \approx \|S^{\Delta t}u_0-S^{\Delta t}_{\varepsilon}u_0\|_{L^2}$
in the numerical tests,
with theoretically estimated $\Delta t^\star$ following (\ref{tstar2}).
Table \ref{kpp_delta_star} summarizes these results
where computation of estimated $\Delta t^\star$ in (\ref{tstar2})
is given by the computation of $M_1$ and $M_2$ with Maple$^\copyright$
according to (\ref{Mone}) and (\ref{Mtwo}).
A really good agreement can be observed even though theoretical results 
underestimate the real values.
The loss of order 
depicted by the numerical results,
is due to the influence of spatial
gradients in the solution,
as it was proven
in \cite{Descombes07}.
This explains the error of the predicted critical $\Delta t^\star$
in (\ref{tstar2})
whenever one gets close to the order loss region.

Numerical results show also that
$\|S^{\Delta t}u_0-S^{\Delta t}_{\varepsilon}u_0\|_{L^2}\propto \varepsilon$
according to (\ref{s_seps}) and consequently,
$\Delta t^\star \propto \varepsilon$;
therefore, the working region or domain of application of the method,
$\Delta t < \Delta t^\star$,
depends directly on the choice of $\varepsilon$
as it can be seen in Table \ref{kpp_delta_star}.
Finally,
in the context of traveling waves,
these numerical experiments show that
$\Delta t^\star \propto k^{-1} \propto 1/\| \partial u_0/\partial x \|_\infty$
according to Table \ref{kpp_delta_star};
hence,
application domains are reduced for stiffer configurations
but numerical results show also that smaller time steps are required
for the same level of accuracy.
These conclusions are easily extrapolated to more general self-similar propagating
waves.
\begin{table}[!htb]
\caption{KPP equation. 
Comparison between 
real $\Delta t^\star_{\mathrm{real}}$,
obtained when
$\|T^{\Delta t}u_0-S^{\Delta t}u_0\|_{L^2} \approx \|S^{\Delta t}u_0-S^{\Delta t}_{\varepsilon}u_0\|_{L^2}$
in the numerical tests,
and theoretically estimated $\Delta t^\star_{\mathrm{est}}$ following (\ref{tstar2}).}
\label{kpp_delta_star}
\begin{center}
\begin{tabular}{ll|ccc}
\hline\noalign{\smallskip}
& & $\varepsilon = 0.05 $& $\varepsilon = 0.005 $& $\varepsilon = 0.0005 $ \\ 
\noalign{\smallskip}
\hline
\noalign{\smallskip}
$k=1$ & $\Delta t^\star_{\mathrm{real}}$ & $2.783$ & $0.1274$ & $1.17\times10^{-2}$ \\
      & $\Delta t^\star_{\mathrm{est}}$ & $1.107$ & $0.1107$ & $1.11\times10^{-2}$ \\
\hline
$k=10$ & $\Delta t^\star_{\mathrm{real}}$ & $0.2803$ & $1.29\times10^{-2}$ & $1.19\times10^{-3}$ \\
      & $\Delta t^\star_{\mathrm{est}}$ & $0.1107$ & $1.11\times10^{-2}$ & $1.11\times10^{-3}$ \\
\hline
$k=100$ & $\Delta t^\star_{\mathrm{real}}$ & $4.33\times10^{-2}$ & $2.12\times10^{-3}$ & $1.92\times10^{-4}$ \\
      & $\Delta t^\star_{\mathrm{est}}$ & $1.11\times10^{-2}$ & $1.11\times10^{-3}$ & $1.11\times10^{-4}$ \\
\end{tabular}
\end{center}
\end{table}

\FloatBarrier 

\section{Construction of the Numerical Strategy}\label{Strategy}
We have presented 
in $\S~$\ref{AdaptSec},
a time adaptive numerical scheme 
fully based on theoretical error estimates
developed in $\S~$\ref{analyse}. We have also studied
the necessary general conditions in order to guarantee an
effective error control based on local error estimates.
In particular,
this has been  shown in the case of reaction traveling waves 
in $\S~$\ref{KPP},
for which
theoretical studies give us some insight into the PDE solution.
Nevertheless, this is not always possible and it is usually 
difficult to carry out such kind of analysis for more realistic models.
Therefore, 
based on the theoretical analysis and
previous illustrations on the influence of the various
parameters of the scheme,
a general numerical procedure 
that completes
the adaptive scheme defined in $\S~$\ref{AdaptSec},
is introduced in the following.

In a first part, we will settle the theoretical framework 
and the numerical procedure needed to estimate $t^\star$,
and to define the appropriate $\varepsilon$.
This will be illustrated by numerical tests performed
on a more complex model of time-space stiff propagating
waves.
These theoretical and numerical studies will
allow to define, at the end, a final numerical strategy.

\subsection{Numerical procedure to estimate critical $t^\star$ and $\varepsilon$}\label{tstar_eps}
Let us consider general system (\ref{sys_rea_dif_u}),
based on theoretical estimates (\ref{estim2}) and (\ref{s_seps}),
we can write 
\begin{equation}\label{error_C0}
S^{\Delta t}u_0-T^{\Delta t}u_0
= C_0 \Delta t^3,
\end{equation}
where $C_0=C_1(u_0)+\mathcal{O}( \Delta t^4)$, and
\begin{equation}\label{error_C_eps}
S^{\Delta t}u_0-S^{\Delta t}_{\varepsilon}u_0
= \varepsilon C_{\varepsilon} \Delta t^2,
\end{equation}
where $C_{\varepsilon}=C_2(u_0)+\mathcal{O}(\varepsilon,\Delta t^3)$;
the dependence of $C_{\varepsilon}$ on $\varepsilon$ is only given in the higher order terms
and it is thus neglected.

For a given $\varepsilon$,
in the same spirit as Corollary \ref{corostar},
 we search for a critical $\Delta t^\star$ such that
\begin{equation}\label{cond_err}
\left\|S^{\Delta t}u_0-T^{\Delta t}u_0\right\|
\leq
\left\|S^{\Delta t}u_0-S^{\Delta t}_{\varepsilon}u_0\right\|
\end{equation}
for all $\Delta t \leq \Delta t^\star$.
According to (\ref{error_C0}) and (\ref{error_C_eps}), we have then the following estimate:
\begin{equation}\label{dt_star}
\Delta t^\star \approx \displaystyle \frac{\varepsilon C_{\varepsilon}}{C_0}.
\end{equation}
For a given $\varepsilon$,
this gives an upper bound for the time steps for which the local error estimate, 
$err=\|S^{\Delta t}u_0-S^{\Delta t}_{\varepsilon}u_0\|$,
is properly estimating
the real Strang local error,
$\left\|S^{\Delta t}u_0-T^{\Delta t}u_0\right\|$,
following (\ref{cond_err}).

In particular, when 
$\Delta t \to \Delta t^\star$, 
we have that 
$err 
\approx
\left\|S^{\Delta t}u_0-T^{\Delta t}u_0\right\|$,
and the local error estimate is predicting more accurately
the real error of integration.
The critical time step, 
$\Delta t^\star$, is directly related to $\varepsilon$
through (\ref{dt_star})
as we have already shown in the previous numerical results
in $\S~$\ref{KPP_res}.
Therefore, 
a suitable
$\varepsilon$ 
will define a critical 
$\Delta t^\star$
such that 
the estimated splitting time steps
$\Delta t$ for a given tolerance $\eta$
will be close enough to critical $\Delta t^\star$,
in order to avoid an excessive overestimation
of the Strang local error and thus,
larger 
time steps can be chosen for a given accuracy tolerance $\eta$.

In order to compute $\Delta t^\star$
for a given $\varepsilon$, we must first 
estimate $C_0$ in (\ref{dt_star}), since $C_\varepsilon$ is computed out of 
the local error estimate, $err$, for known $\Delta t$ and $\varepsilon$ in (\ref{error_C_eps}).
Estimating $C_0$ amounts to directly estimate Strang local error through (\ref{error_C0})
and thus, the accuracy of the simulation might be controlled in this way without
relying on a local error estimate as proposed in the embedded method strategy
in $\S~$\ref{AdaptSec}.
Nevertheless, as we will see in the following,
in order to estimate 
$C_0$ and the Strang local error,
we must define new local estimators and a numerical procedure
that becomes rapidly very expensive if we want to implement
such error control technique.
Therefore, we must rely on a local error estimate given by
a less expensive strategy for which the computation of 
$C_0$ is only performed from time to time
to guarantee the validity of local error estimates. 

The next Lemma will be useful to define the numerical procedure
to estimate $C_0$.
\begin{lemma}\label{lemma}
Let us consider system (\ref{sys_rea_dif_u})
and assume
a local Lipschitz condition for $f$:
\begin{equation}\label{error_Lipf}
\left\|f(u)-f(v)\right\|
\leq \lambda
\left\|u-v\right\|.
\end{equation}
For a finite $\Delta t$ the following holds
\begin{equation}\label{error_Lip}
\left\|T^{\Delta t}u_0-T^{\Delta t}v_0\right\|
\leq \omega
\left\|u_0-v_0\right\|,
\end{equation}
with $\omega = 1 + \kappa \Delta t$
for small enough $\Delta t$.
\end{lemma}
\begin{proof}
Using Duhamel's formula for (\ref{sys_rea_dif_u}) yields
\begin{equation}\label{duhamel}
T^t u_0 - T^t v_0 =
e^{t\partial^2_x}(u_0-v_0) + \int _0 ^t e^{(t-s)\partial^2_x}
\left(f(T^s u_0) - f(T^s v_0)\right) \, {\rm d} s.
\end{equation}
Taking norms and applying recursively (\ref{duhamel}),
\begin{eqnarray}\label{lambda}
\left\| T^t u_0 - T^t v_0\right\| & \leq &
\left\| u_0-v_0 \right\| +
\lambda \int _0 ^t \| T^s u_0 - T^s v_0\| \, {\rm d} s, \nonumber \\
& \leq & e^{\lambda t} \| u_0-v_0 \|,
\end{eqnarray}
proves (\ref{error_Lip}) for $t=\Delta t$ finite.
\end{proof}

If we define a local estimator,
$e_1= S^{a_1\Delta t}u_0-S^{b_1\Delta t}(S^{c_1\Delta t}u_0)$,
such that $a_1=b_1+c_1$,
we obtain that
\begin{eqnarray}\label{error_dt_2}
S^{b_1\Delta t}(S^{c_1\Delta t}u_0)-T^{a_1\Delta t}u_0
&=&
S^{b_1\Delta t}(S^{c_1\Delta t}u_0) - T^{b_1\Delta t}(S^{c_1\Delta t}u_0)\nonumber\\
& & +T^{b_1\Delta t}(S^{c_1\Delta t}u_0) - T^{b_1\Delta t}(T^{c_1\Delta t}u_0),
\nonumber\\
&=&
C_{S^{c_1\Delta t}u_0} b_1^3\Delta t^3 \nonumber\\
& &+T^{b_1\Delta t}(S^{c_1\Delta t}u_0) - T^{b_1\Delta t}(T^{c_1\Delta t}u_0),
\end{eqnarray}
where $C_{S^{c_1\Delta t}u_0}=C_1(S^{c_1\Delta t}u_0)+\mathcal{O}( \Delta t^4)$.
Therefore,
assuming that $C_{S^{c_1\Delta t}u_0}\approx C_0$
and considering Lemma \ref{lemma}, 
it
follows from the difference between
(\ref{error_C0}) at $a_1\Delta t$
and (\ref{error_dt_2}): 
\begin{eqnarray}\label{est_e1}
\|e_1-(a_1^3-b_1^3)C_0\Delta t^3\|
&\leq&
\omega
\|T^{c_1\Delta t}u_0-S^{c_1\Delta t}u_0\|,
\nonumber\\
&\leq&
\omega C_0c_1^3\Delta t^3.
\end{eqnarray}

Hence, 
defining a second local estimator,
$e_2= S^{a_2\Delta t}u_0-S^{b_2\Delta t}(S^{c_2\Delta t}u_0)$,
such that $a_2=b_2+c_2$, we obtain a second expression
similar to (\ref{est_e1}) with $e_2$ and $(a_2,b_2,c_2)$,
and we can estimate $C_0$ and $\omega$.
In particular, we 
notice that $b_1$ should be close to $b_2$ in order to better approximate
$\omega$ into (\ref{error_Lip}) and (\ref{est_e1}),
and that $c_1$ and $c_2$ should also be small enough to 
guarantee 
$C_{S^{c_1\Delta t}u_0}\approx C_0$ and $C_{S^{c_2\Delta t}u_0}\approx C_0$.
On the other hand, 
to optimize the required number of extra computations
from a practical point of view, 
we can use
the estimator $e_2$ to compute estimator $e_1$ by 
setting $a_2=c_1$, and we can also fix $a_1=1$ so we can use $S^{a_1\Delta t}u_0$ 
for the time integration of the problem.
In this way, the extra computations 
needed to compute
local estimators $e_1$ and $e_2$
will be given by 
$S^{c_2\Delta t}u_0$,
$S^{b_2\Delta t}(S^{c_2\Delta t}u_0)$,
$S^{c_1\Delta t}u_0$ and $S^{b_1\Delta t}(S^{c_1\Delta t}u_0)$
within a time step 
$\Delta t$.
Then, 
we will be able to compute $\omega$ and $C_0$,
by solving two expressions of type (\ref{est_e1}).
The next numerical example illustrates the validity of this numerical procedure.

\subsection{Numerical example of evaluation of critical $t^\star$: BZ equation}\label{BZ}
We are concerned with the numerical approximation of a model of 
the Belousov-Zhabotinski reaction, a catalyzed oxidation of an organic species by acid 
bromated ion (for more details and illustrations, see \cite{Epstein98}).
We thus consider the model introduced in \cite{Scott94} and coming from the classic work of 
Field, Koros and Noyes (FKN) (1972), 
which takes into account three species: $\mathrm{HBrO_2}$ (hypobromous acid), 
bromide ions $\mathrm{Br^-}$  and cerium(IV). 
Denoting by $a=[\mathrm{Ce(IV)}]$, $b=[\mathrm{HBrO_2}]$ and $c=[\mathrm{Br^-}]$, 
we obtain a very stiff system of three partial dif\-fe\-ren\-tial equations:
\begin{equation} \label{bz_eq_3var_diff}
\left.
\begin{array}{rcl}
\partial _t a - D_a\, \partial^2_x a&=&\displaystyle \frac{1}{\mu}(-qa-ab+fc),\\[2ex]
\partial _t b\, - D_b\, \partial^2_x b&=&\displaystyle 
\frac{1}{\epsilon}\left(qa-ab+b(1-b)\right),\\[2ex]
\partial _t c\, - D_c\, \partial^2_x c&=&b-c,
\end{array}
\right\}
\end{equation}
with diffusion coefficients  $D_a$, $D_b$ and $D_c$, and some real positive
parameters $f$, small $q$, and small $\epsilon$, 
$\mu$, such that  $\mu \ll \epsilon$.

The dynamical system associated with this system
models reactive excitable media with a
large time scale spectrum (see \cite{Scott94} for more details). 
Moreover, the spatial configuration with
addition of diffusion generates propagating wavefronts
with steep spatial gradients.
Hence, this model presents all the difficulties associated
with a stiff time-space multi-scale configuration.
The advantages of applying a splitting strategy to
these models have already been studied and presented in \cite{Dumont03}.

We consider the 1D application of problem (\ref{bz_eq_3var_diff})
with homogeneous Neumann boundary conditions in
a space region of
$[0,80]$ with a spatial discretization of $4001$ points,
good enough to prevent important spatial discretization errors,
and the following parameters, taken from \cite{Scott94}: 
$\epsilon = 10^{-2}$, $\mu = 10^{-5}$, $f=3$ and $q=2\times 10^{-4}$,
with diffusion coefficients $D_a=1$, $D_b=1$ 
and $D_c=0.6$.
Reference solution and Strang approximations are defined in the same way
as in the KPP application with the same tolerances for the time
integration solvers.

First of all, we validate
theoretical order estimates (\ref{estim1}), (\ref{estim2}) and (\ref{s_seps})
and verify relation (\ref{somme_err}).
Figure \ref{bz_error}
shows $L^2$ errors between $T^t u_0$, $S^t u_0$ and $S^t_\varepsilon u_0$
solutions for several $\varepsilon$ and 
the real $\Delta t^\star$
such that 
$\|T^{\Delta t}u_0-S^{\Delta t}u_0\|_{L^2} \approx \|S^{\Delta t}u_0-S^{\Delta t}_{\varepsilon}u_0\|_{L^2}$,
obtained after treating the numerical results. Maximum $L^2$ error considers the maximum value between
normalized local errors for 
$a$, $b$ and $c$ variables; in these numerical tests, it corresponds usually to variable $b$.
\begin{figure}[!htb]
 \begin{center}
 \includegraphics[width=0.45\hsize]{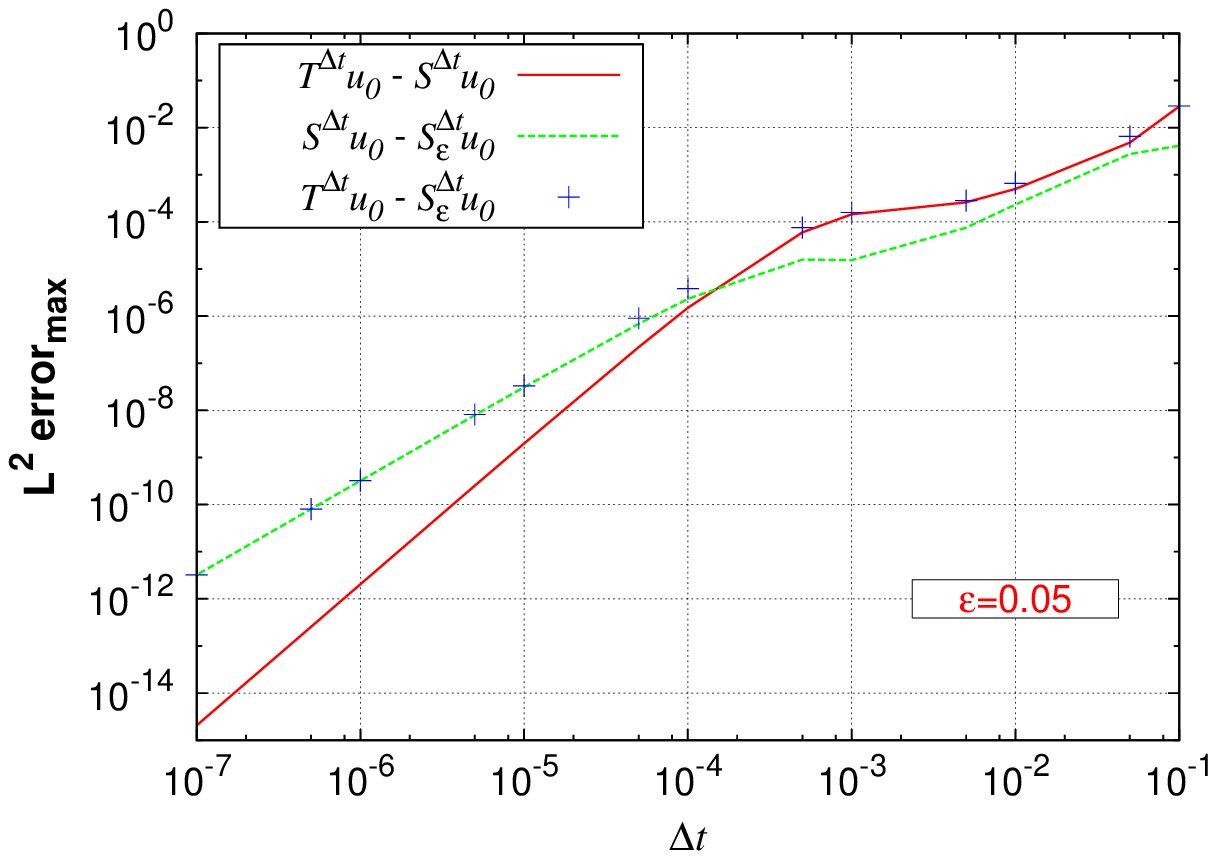}
 \includegraphics[width=0.45\hsize]{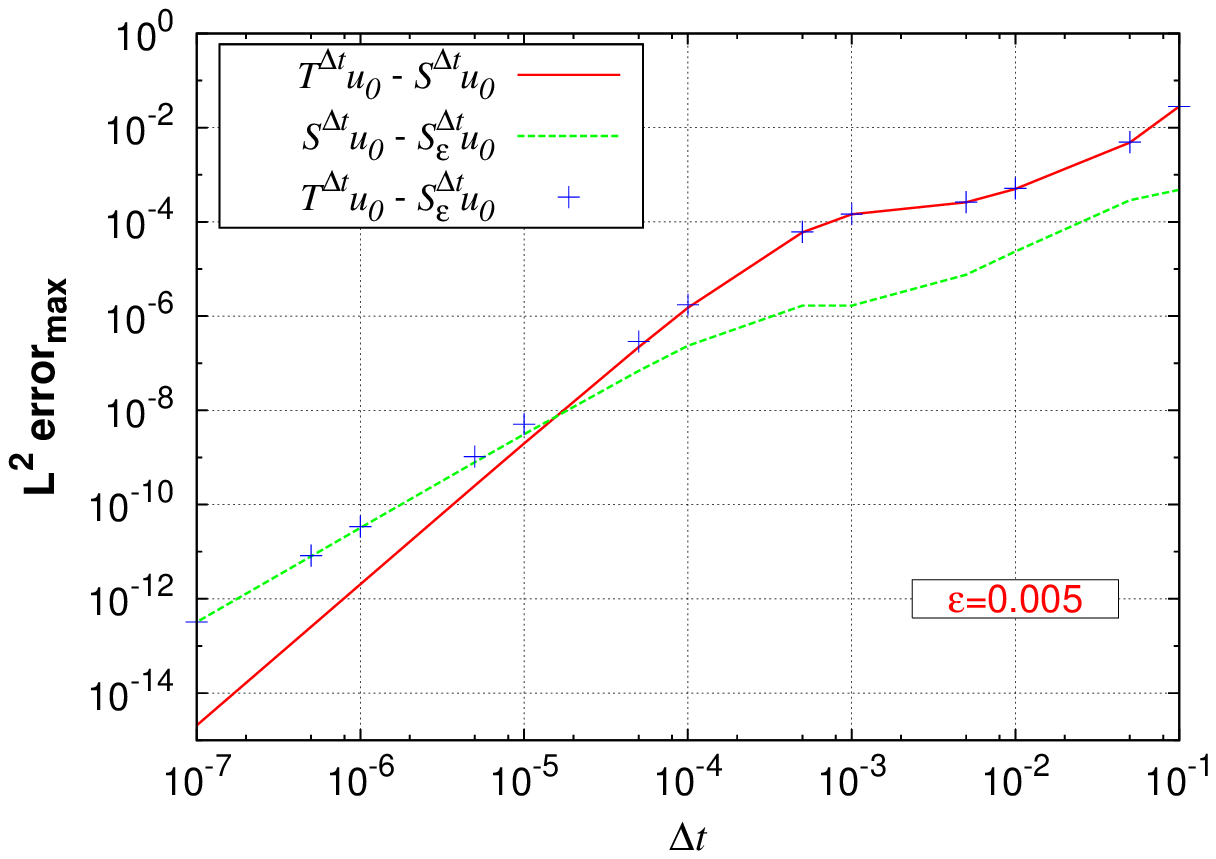}
 \includegraphics[width=0.45\hsize]{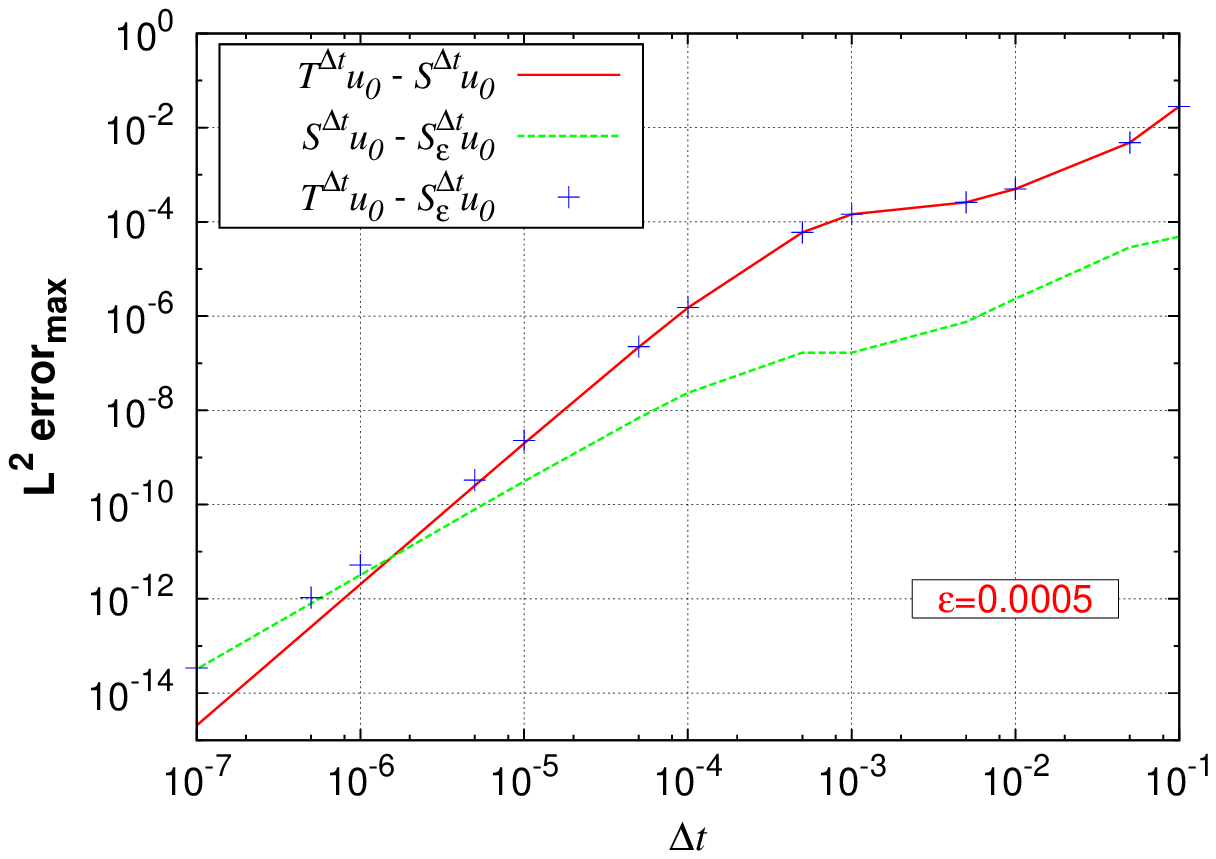}
 \includegraphics[width=0.45\hsize]{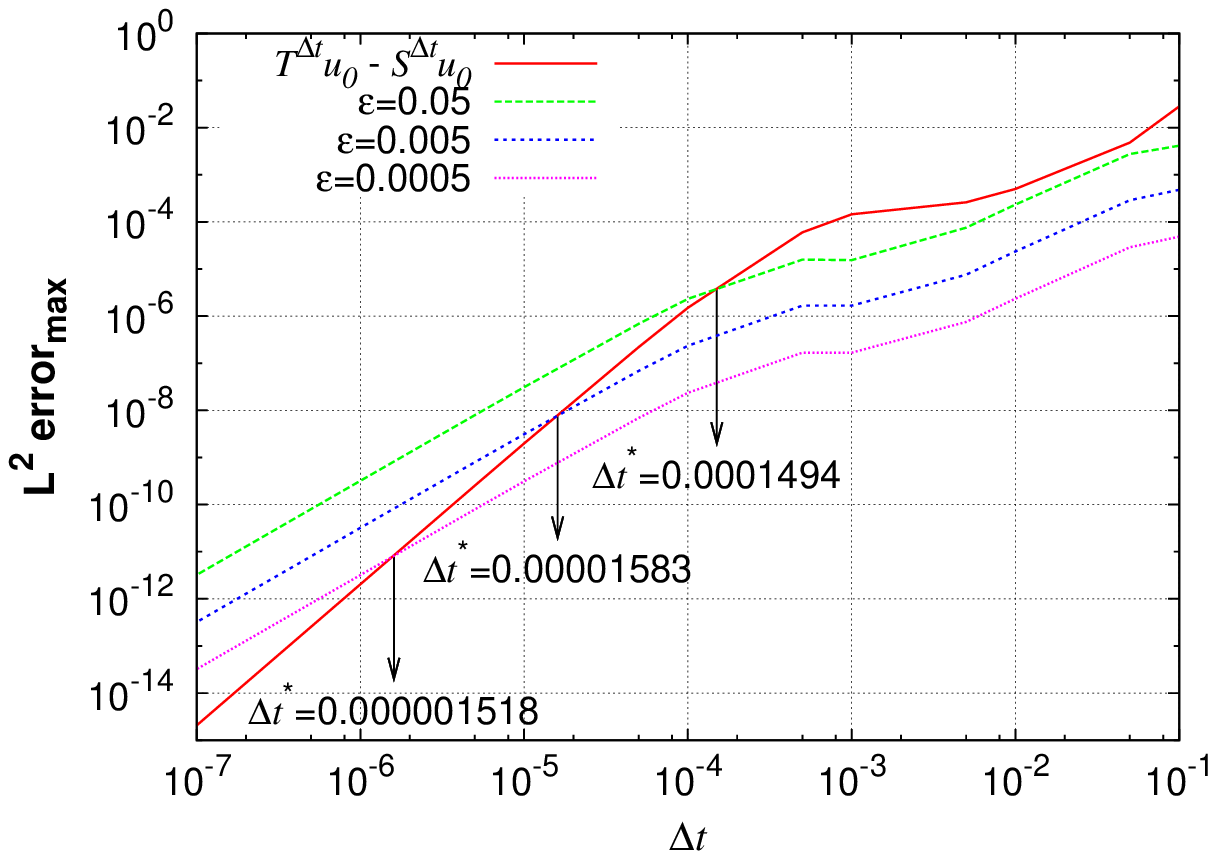} 
   \caption{BZ equation.
Maximum local $L^2$ errors for several splitting time steps $\Delta t$ and
$\varepsilon=0.05$ (top left), $0.005$ (top right) and $0.0005$ (bottom left).
 Bottom right: critical splitting time steps $\Delta t^\star$
obtained when $\|T^{\Delta t}u_0-S^{\Delta t}u_0\|_{L^2} \approx \|S^{\Delta t}u_0-S^{\Delta t}_{\varepsilon}u_0\|_{L^2}$
in the numerical tests. \label{bz_error}}
 \end{center}
\end{figure}

Let us now define the two sets 
$(a_1,b_1,c_1)$ and $(a_2,b_2,c_2)$,
and compute local estimators $e_1$ and $e_2$ in order to 
obtain $C_0$ according to (\ref{est_e1}) with $\Delta t=\Delta t_0=10^{-5}$;
that is a time step for which there is no order loss yet,
as seen in Figure \ref{bz_error}.
As it was previously detailed, we consider $a_1=1$
and $a_2=c_1$ to avoid some extra computations.
Furthermore, $b_2$ should be close to $b_1$, and
$c_1$ and $c_2$ small enough.
Setting $b_1$ larger than $1/2$ would yield 
more different $b_2$ since $c_1=a_2$.
On the other hand, for $b_1$ smaller than
$1/2$ we can even set $b_2=b_1$ but in this case
$c_1$ will be larger than $1/2$.
Therefore, we reach a compromise by setting
$b_1=1/2$ that yields $c_1=a_2=1/2$,
so we can choose for instance 
$b_2=2/5$ close to $b_1$, and thus,
$c_2=1/10$.

With the local error estimate, $err=\|S^{\Delta t}u_0-S^{\Delta t}_{\varepsilon}u_0\|_{L^2}$, 
for the various time steps
and several $\varepsilon$
shown in Figure \ref{bz_error},
Figure \ref{bz_est} presents the estimated critical $\Delta t^\star$ calculated 
with (\ref{dt_star})
from the estimated $C_0(\Delta t_0)$ and $err$.
These critical time steps, $\Delta t^\star$, 
estimated with (\ref{dt_star}) are in good agreement with numerically
measured $\Delta t^\star$ in Figure \ref{bz_error},
and depend on the value of $\varepsilon$.
Hence, the domain of application or working region of the method,
$\Delta t\leq \Delta t^\star$, might be settled depending
on the desired level of accuracy by means of an appropriate
choice of $\varepsilon$.
For instance, if we consider
the case $\varepsilon=0.05$ in Figure \ref{bz_error},
for $\Delta t=10^{-6}$, the local error estimate is given by $err \approx 10^{-10}$
whereas the real Strang local error is $\sim 10^{-12}$.
This overestimation of the local error will certainly imply
an underestimation in the required size of the time steps for a given tolerance.
Therefore, for a given tolerance $\eta$ a more suitable configuration
should consider an $\varepsilon$ such that $\Delta t \approx \Delta t^\star$
in order to reduce excessive overestimations of local errors.
\begin{figure}[!htb]
 \begin{center}
\includegraphics[width=0.45\hsize]{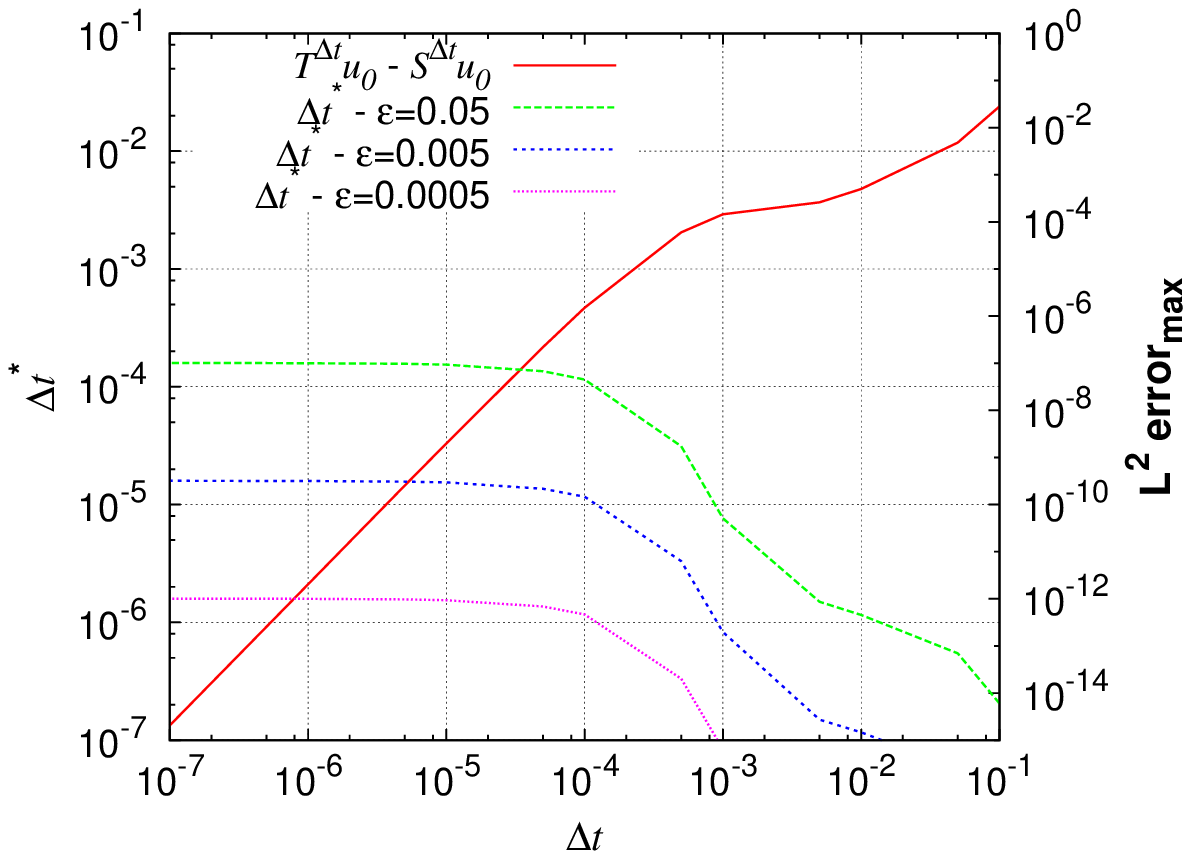} 
 \includegraphics[width=0.45\hsize]{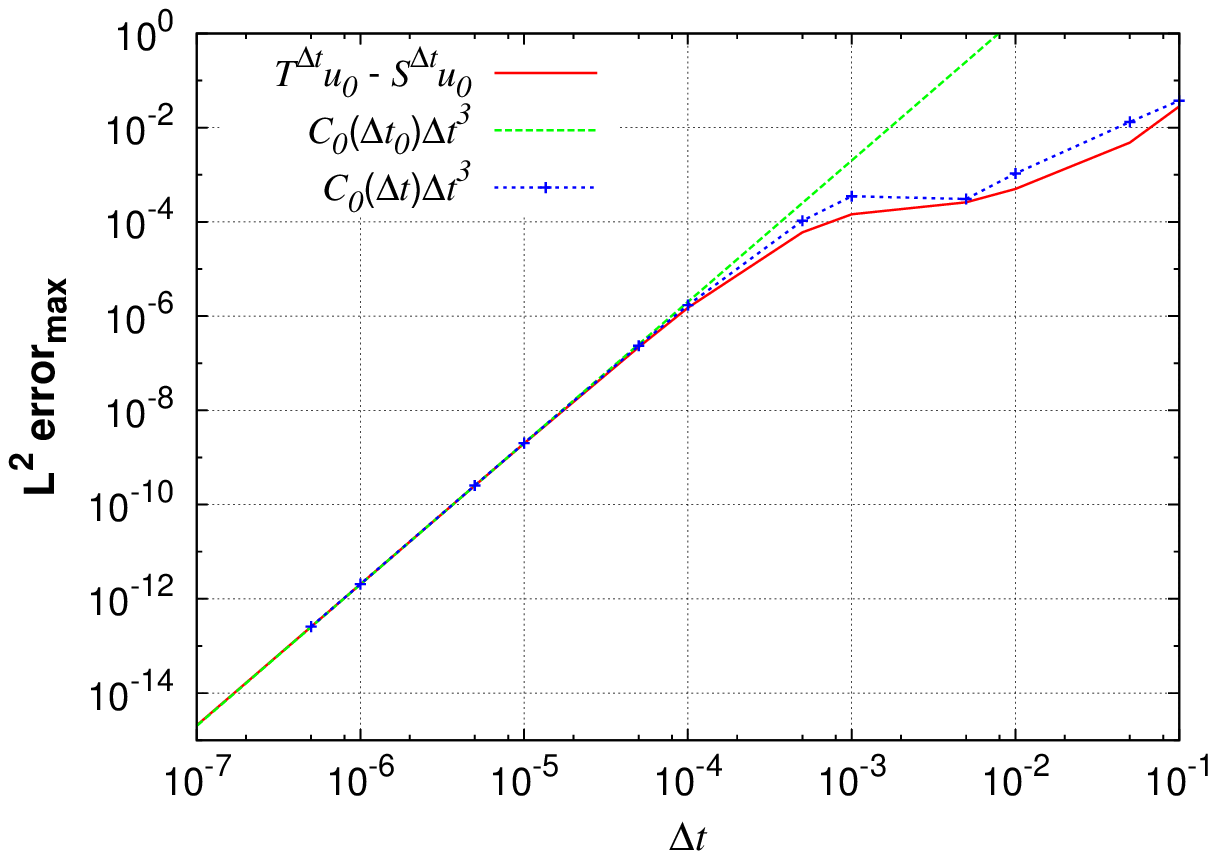}
   \caption{BZ equation.
Working region of the method $\Delta t\leq \Delta t^\star$ with 
$\Delta t^\star$ calculated with $C_0$ estimated at $\Delta t_0=10^{-5}$ and $err$ obtained for 
several splitting time steps $\Delta t$ and $\varepsilon$ (left).
Right:
predicted Strang error calculated with $C_0$ estimated at $\Delta t_0=10^{-5}$ and 
locally at
several splitting time steps $\Delta t$.}
\label{bz_est}
 \end{center}
\end{figure}

In the illustration shown 
in Figure \ref{bz_est}, 
$C_0$ was estimated in the third order region
of the method and therefore, all values are well approximated as long as $\Delta t$
remains in this region. In particular, critical $\Delta t^\star$ will be progressively
underestimated for larger $\varepsilon$ and consequently, it will impose
smaller time steps for a given tolerance; this is already the case for
$\varepsilon=0.05$, for which $\Delta t^\star$ is in the transition zone towards
the lower order region.
Even though the computation of $C_0$ with 
small time steps will be less expensive,
a much more accurate procedure considers current time step as shown
in Figure \ref{bz_est}.
In particular,
by estimating locally $C_0$, we are estimating real
Strang error and thus, $\Delta t \leq \Delta t^\star$
guarantees prescribed accuracy even if 
asymptotic order estimates are no longer verified.
This allows
to properly extend the domain of application over
the whole range of possible time steps for a given 
accuracy; an
extremely important issue
for real applications
for which splitting time steps may
go far beyond asymptotic behavior
including the potential order reduction
region associated with the stiffness of the problem.

\FloatBarrier

\subsection{Numerical strategy}\label{num_str}
Previous studies conducted in $\S~$\ref{tstar_eps}
and $\S~$\ref{BZ}
allow to properly complete the adaptive splitting strategy 
introduced in $\S~$\ref{AdaptSec}.
In this part we conduct the final description of the numerical
strategy.

Let us consider general problem (\ref{sys_reac_diff_gen}) for
$\mathbf{u} \in \mathbb{R}^{m}$, for which we use $S_2^t$ in
(\ref{strang_u}) as resolution scheme.
Depending on the problem, the adaptive method will be applied 
considering time evolution of 
$l \leq m$ variables:
$\mathbf{\tilde{u}} \in \mathbb{R}^{l}$.
Let us denote $\Omega _l$ the set of indices of these variables.
In order to consider only $l<m$ variables, the formers must be
decoupled of the remaining $m-l$ variables in the reactive term $\mathbf{f}(\mathbf{\tilde{u}})$ in (\ref{sys_reac_diff_gen}).
To simplify the presentation, we will only consider $\varepsilon \in (0,\varepsilon_{max})$,
$\varepsilon_{max}<1/2$.

We set the accuracy tolerance $\eta$,
an initial time step $\Delta t^0$ and initial $\varepsilon _0$,
and perform the time integration of (\ref{sys_reac_diff_gen}) 
with the Strang scheme $S_2^t$ and the embedded
shifted one $S^t_{2,\varepsilon}$ given by
(\ref{strangdec}).
We compute local error estimate $err$ and new time step 
$\Delta t^{\rm new}$ according to (\ref{delta_split_ef}).
If $err$ is smaller than $\eta$, current time step solution
is accepted and simulation time evolves; otherwise,
current solution is rejected and the time integration is 
recomputed 
with $\Delta t^{\rm new}$.
In particular, 
it is better to choose rather small $\Delta t^0$ 
to avoid initial rejections.

In order to guarantee an effective error control, we
define the working region $\Delta t \leq \Delta t^\star$
by estimating the corresponding $\Delta t^\star$
for current $\varepsilon$.
This is done for the first time step $\Delta t^0$
and then 
periodically after $N$ accepted time steps
depending on the problem,
based on the numerical 
procedure introduced in $\S~$\ref{tstar_eps}.
Computation of critical $\Delta t^\star$ is also performed with $\mathbf{\tilde{u}}$,
and a rather large initial $\varepsilon _0$ is suitable
to initially guarantee $\Delta t \leq \Delta t^\star$.

We define then a suitable working region 
$\Delta t \in [\beta \Delta t^\star, \gamma \Delta t^\star]$
with $0<\beta<\gamma \leq 1$,
for which 
splitting time steps are close to $\Delta t^\star$.
A new $\varepsilon$ is then computed if 
$\Delta t$ is much lower than $\Delta t^\star$ ($\Delta t<\beta \Delta t^\star$)
in order to avoid unnecessary small time steps;
or if $\Delta t$ is very close or possibly larger than $\Delta t^\star$ ($\Delta t >\gamma \Delta t^\star$)
with $\gamma$ close to one, in order to increase upper bound of the domain of application.
This guarantees that $\varepsilon$ is dynamically computed and properly adapted
to the dynamics of the phenomenon.

Finally,
the numerical resolution strategy can be summarized as
follows,
where $\mathbf{U} \in \mathbb{R}^{m\times n}$ stands for the spatial discretization of
$\mathbf{u}$ over $n$ points, $\mathbf{U} := (u^{(j,k)})$ such that $j\in [1,m]$
and $k \in [1,n]$.

\begin{itemize}
\item {\bf Input parameters.} Define accuracy tolerance $\eta$, time domain of study
$[t_0,T]$, initial time step $\Delta t^0$, initial $\varepsilon _0$, and period of
computation of $\Delta t^\star$: $N$. 
\item {\bf Initialization.} Set iteration counter $i=0$ and $t=t_0$,
$\mathbf{U}=\mathbf{U_0}$, $\Delta t = \Delta t^0$, $\varepsilon = \varepsilon _0$.
We define a flag $estimate$ initialized as {\tt .false.}.
Throughout the whole computation, we need to store $\mathbf{U}$, {\bf an array of size $m\times n$}.
\item {\bf Time evolution.} If $t<T$:
\begin{enumerate}
 \item 
Only if $\displaystyle \frac{i}{N}=\left\lfloor \frac{i}{N} \right\rfloor$ or $estimate$ is {\tt .true.}:\\
{\bf Computation of critical $\Delta t^\star$ I:} For the sets $(a_1,b_1,c_1)$ and $(a_2,b_2,c_2)$
with $a_1=1$ and $a_2=c_1$, we compute successively:
 \begin{itemize}
  \item  $\mathbf{\tilde{U}_1}=S^{c_2\Delta t} \mathbf{\tilde{U}_0}$, where $\mathbf{\tilde{U}_0}$ is 
built out of $\mathbf{U}$, $\mathbf{\tilde{U}_0} = (u^{(j,\cdot)})_{j\in \Omega _l}$;
  \item  $\mathbf{\tilde{U}_1}=S^{b_2\Delta t} \mathbf{\tilde{U}_1}$;
  \item  $\mathbf{\tilde{U}_2}=S^{c_1\Delta t} \mathbf{\tilde{U}_0}$;
  \item  $e_1= \max _{j\in\Omega _l}\|\tilde{u}_2^{(j,\cdot)} - \tilde{u}_1^{(j,\cdot)}\|$;
  \item  $\mathbf{\tilde{U}_2}=S^{b_1\Delta t} \mathbf{\tilde{U}_2}$;
  \item  $estimate$ is set to {\tt .true.}.
   \end{itemize}
These operations needs to store $\mathbf{\tilde{U}_1}$ and $\mathbf{\tilde{U}_2}$, 
{\bf two arrays of size $l\times n$}.

\item {\bf Time integration over $\Delta t$:}
We compute successively:
\begin{itemize}
 \item for each $k \in [1,n]$, $u_{new}^{(\cdot,k)}=Y^{\Delta t/2} u^{(\cdot,k)}$;
 \item for each $k \in [1,n]$, $\tilde{u}_1^{(\cdot,k)}=Y^{\varepsilon \Delta t} \left.u_{new}^{(j,k)}\right|_{j\in \Omega _l}$;
 \item $\mathbf{U_\star}=X^{\Delta t} \mathbf{U_\star}$, with $\mathbf{U_\star}=\phantom{x}^t (\mathbf{U_{new}},\mathbf{\tilde{U}_1})$;
 \item for each $k \in [1,n]$, $u^{(\cdot,k)}_\star=Y^{(1/2-\varepsilon)\Delta t} u^{(\cdot,k)}_\star$;
 \item for each $k \in [1,n]$, $u^{(\cdot,k)}_{new}=Y^{\varepsilon \Delta t} u^{(\cdot,k)}_{new}$;
 \item  $err= \max _{j\in\Omega _l}\|\tilde{u}_{new}^{(j,\cdot)} - \tilde{u}_1^{(j,\cdot)}\|$.

\end{itemize}
We need to store $\mathbf{U_{new}}$, {\bf an array of size $m\times n$}.

\item 
Only if $estimate$ is {\tt .true.}:\\
{\bf Computation of critical $\Delta t^\star$ II:}
We compute successively:
\begin{itemize}
 \item $e_2= \max _{j\in\Omega _l}\|\tilde{u}_{new}^{(j,\cdot)} - \tilde{u}_2^{(j,\cdot)}\|$;
 \item $C_0$ using (\ref{est_e1}) with $e_1$ and $e_2$;
 \item estimate $\Delta t^\star$ out of (\ref{dt_star}) and set $\Delta t^\star=\zeta \Delta t^\star$
with security factor $0<\zeta \leq 1$ close to one;
  \item  $estimate$ is set to {\tt .false.}.
\item If $\Delta t \notin [\beta \Delta t^\star, \gamma \Delta t^\star]$ with $0<\beta<\gamma \leq 1$:
$estimate$ is set to {\tt .true.}.
\end{itemize}

\item 
Only if $estimate$ is {\tt .true.} and $i > 0$:\\
{\bf Computation of $\varepsilon$:}
According to (\ref{dt_star}) with $err$, $C_0$ and $\Delta t^\star=\Delta t$:
\begin{itemize}
 \item $\varepsilon=\min \{\theta \varepsilon,\varepsilon_{max}\}$ with $\theta \geq 1$ as security factor;
\item computation of  $\Delta t^\star$ with new $\varepsilon$;
\item $estimate$ is set to {\tt .false.}.

\end{itemize}

\item {\bf Computation $\Delta t^{new}$:}
According to (\ref{delta_split_ef}) with security factor $0<\upsilon \leq 1$ close to one.
\begin{itemize}
\item If $\Delta t > \Delta t^\star$: set $err=tol + C$ with $C>1$. Used to potentially
reject initial $\Delta t=\Delta t^0$.
 \item If $\Delta t^{new}>\Delta t^\star$ and $\varepsilon \neq \varepsilon _{max}$: $estimate$ is set to {\tt .true.}.
\item $\Delta t = \min \{\Delta t^{new},\Delta t^\star\}$.
\item If $err \leq tol$: $t=t+\Delta t$, $i=i+1$, $\Delta t = \min \{\Delta t,T-t\}$
and $\mathbf{U}=\mathbf{U_{new}}$.
\end{itemize}

\end{enumerate}

\end{itemize}

In this strategy, 
reaction is always integrated point by point if the reactive term is modeled by a 
system of ODEs without spatial coupling. This integration can be performed completely in parallel
\cite{article_avc,article_ESAIM}.
On the other hand,
for linear diffusion problems, another alternative considers a variable by variable resolution,
for each $j \in [1,m]\bigcup \Omega_l$:
\begin{equation}
u_{\star}^{(j,\cdot)}=X^{\Delta t} u_\star^{(j,\cdot)},
\end{equation}
that can also be performed in parallel 
\cite{article_avc}.

Depending on the problem, either the
computation of critical $\Delta t^\star$ (steps (1), (3) and (4)),
or the computation of $\varepsilon$ (step (4))
can be potentially removed if one considers large enough $\varepsilon _0$
and fine enough $\eta$.
Finally,
the whole strategy with all steps needs to store at worst
two arrays of size $l\times n$ and other two of size $m\times n$,
beyond memory requirements of diffusion and reaction solvers.

\section{Final Numerical Evaluation of the Method}\label{Appli}
In this last part,
we evaluate the performance of the method in terms of accuracy of the simulation,
and show that an effective control of the simulation error is performed
in the context of two different problems.
First,
we will consider a propagating wave
featuring time-space multi-scale character.
Then, the
potential of the method is fully exploited for a more complex configuration of repetitive
gas discharges generated by high frequency pulsed applied 
electric fields followed by long time
scale relaxation, for which a precise description of discharge and post-discharge
phases is achieved.

\subsection{BZ equation revisited}
Coming back to BZ model, we perform a time integration of
(\ref{bz_eq_3var_diff}) with several accuracy tolerances $\eta$.
First of all, we consider the numerical strategy detailed in $\S~$\ref{num_str}
without taking into account steps (1), (3) and (4), that is 
without computation of neither critical $\Delta t^\star$ nor $\varepsilon$.
We set $\Delta t^0= 10^{-7}$ and $\varepsilon _0=0.05$ in all cases,
with $t\in[0,2]$.
In this example, a rather small initial splitting time step is 
chosen to avoid initial rejections even though this initial
rejection phase usually does not take many steps as it will shown in 
the next example.
On the other hand, we have chosen a intermediary value
for $\varepsilon$ in order to clearly distinguish the 
different behaviors of the strategy in terms of prediction
of the local errors depending on the proposed tolerance.

Figure \ref{bz_step_error} shows time evolution of accepted splitting
time steps $\Delta t$. In this case, BZ equation models a propagating
self-similar wave, so splitting time step stabilizes once the overall phenomenon
is solved within the prescribed tolerance $\eta$. 
Local error estimates $err$ are also shown, which naturally verify
prescribed accuracy, since we impose time steps
for which 
$err$ is limited by $\eta$ through (\ref{delta_split_ef}).
\begin{figure}[!htb]
 \begin{center}
 \includegraphics[width=0.45\hsize]{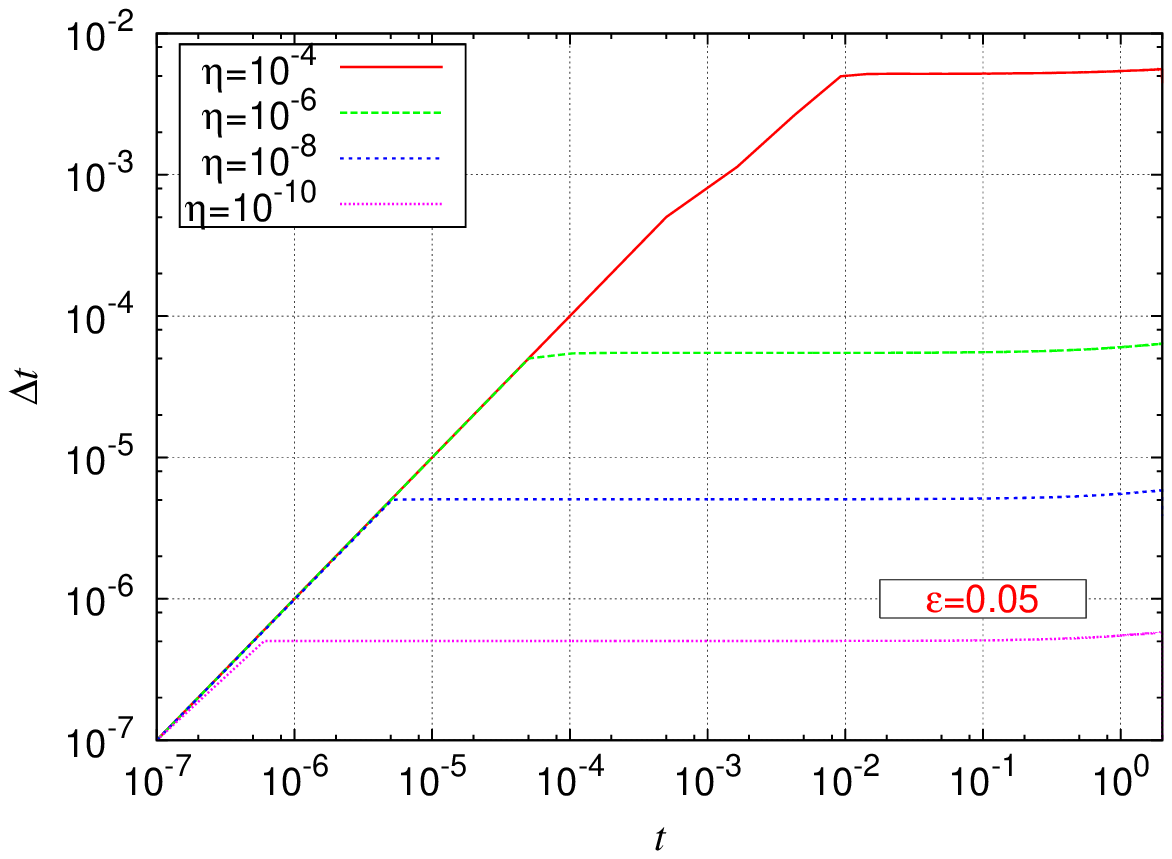}
 \includegraphics[width=0.45\hsize]{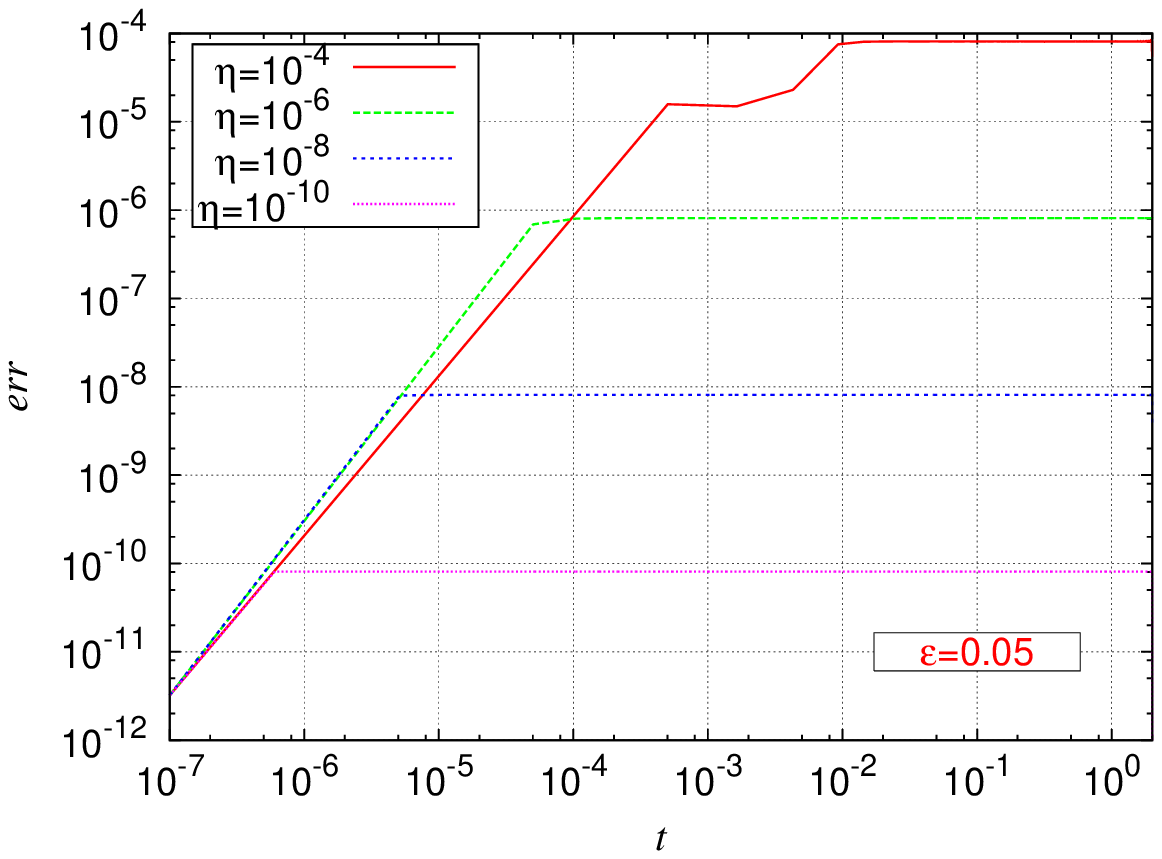} 
   \caption{BZ equation.
Time evolution of
accepted splitting time steps $\Delta t$ (left)
and
local $L^2$ error estimates $err=\|S^{\Delta t}u_0-S^{\Delta t}_{\varepsilon}u_0\|_{L^2}$
(right),
for several tolerances
$\eta$ and $\varepsilon=0.05$.
 \label{bz_step_error}}
 \end{center}
\end{figure}

Table \ref{bz_err_global} summarizes global $L^2$ errors between splitting and reference
solutions at the end of the
time domain of study, $t=2$.
For a fine enough $\eta$ and consequently, small enough time steps,
a precise error control is achieved by the
local error control strategy as we could have expected from previous results
in Figure  \ref{bz_error} for $\varepsilon = 0.05$.
Nevertheless, for $\eta=10^{-4}$ we can see rather high global errors even
if this configuration considers naturally less time integration steps and
thus, less accumulation of local approximation errors.
If we take a look at Figure \ref{bz_error}, we note that for 
$\varepsilon = 0.05$ and local errors of about $10^{-4}$,
the local error estimate, $err$, is not predicting properly real Strang errors, as it was previously
discussed,
since $\Delta t > \Delta t^\star$.
Therefore, a strategy that introduces
a more precise description of errors for 
a larger range of time steps
must be considered, whenever the required accuracy casts the method
away from its asymptotic behavior.
This is an under covered difficulty of any time adaptive technique
based on a lower order embedded method,
and to our knowledge, an open problem that has not been studied much,
and that this work tries to overcome.
\begin{table}[!htb]
\caption{BZ equation. 
$L^2$ errors at final time $t=2$ for $a$, $b$, $c$ variables and
several tolerances $\eta$.}
\label{bz_err_global}
\begin{center}
\begin{tabular}{cccc}
\hline\noalign{\smallskip}
$\eta$& $L^2$ error $a$ & $L^2$ error $b$ & $L^2$ error $c$ \\ 
\noalign{\smallskip}
\hline
\noalign{\smallskip}
$10^{-4}$ & $7.97\times10^{-3}$ & $1.07\times10^{-2}$ & $4.72\times10^{-3}$ \\
$10^{-6}$ & $1.71\times10^{-6}$ & $1.83\times10^{-6}$ & $7.98\times10^{-7}$ \\
$10^{-8}$ & $1.45\times10^{-8}$ & $1.54\times10^{-8}$ & $6.78\times10^{-9}$ \\
$10^{-10}$ & $1.74\times10^{-10}$ & $1.75\times10^{-10}$ & $1.08\times10^{-10}$ \\
\hline
\end{tabular}
\end{center}
\end{table}

Let us now consider the entire strategy with all steps for several
tolerances with $\Delta t^0= 5\times 10^{-7}$ and $\varepsilon _0=0.05$.
In the following illustrations we have considered
the following parameters:
$\varepsilon_{max}= 0.999$;
$a_1=1$, $b_1=c_1=a_2=1/2$, $b_2=2/5$ and $c_2=1/10$ for intermediary time 
steps evaluations;
$\zeta = 0.9$ as security factor of critical $\Delta t^\star$ estimate;
$\beta = 0.1$ and $\gamma = 0.95$ to define the working region 
$\Delta t \in [\beta \Delta t^\star, \gamma \Delta t^\star]$;
$\theta = 10$ as security factor of $\varepsilon$ estimate;
$C=10$ to potentially reject initial time step $\Delta t^0$;
and $\upsilon = 0.9$ as security factor of $\Delta t^{new}$ estimate.
All local estimators, $err$, $e_1$ and $e_2$, are computed with normalized $L^2$ norms.

Considering the propagating phenomenon, we set $N=10$, but we estimate
$\Delta t^\star$ only twice for $i=0$ and $i=N$.
Figure \ref{bz_step_error_corr} shows time evolution of
splitting time steps; there are different scenarios depending on
the required accuracy.
In all cases for $\varepsilon _0=0.05$, we estimate initially
$\Delta t^\star \approx 1.4\times 10^{-4}$.
For $\eta=10^{-4}$, this limitation implies smaller time steps
than what is required for the prescribed tolerance.
Thus, $\Delta t$  increases until $\Delta t^{new} > \Delta t^\star$
and a new $\varepsilon$ is estimated: $\varepsilon \approx 0.43$.
No substantial changes are made when $i=N$, since 
$\Delta t \in [\beta \Delta t^\star, \gamma \Delta t^\star]$
for the current $\eta$.
\begin{figure}[!htb]
 \begin{center}
 \includegraphics[width=0.45\hsize]{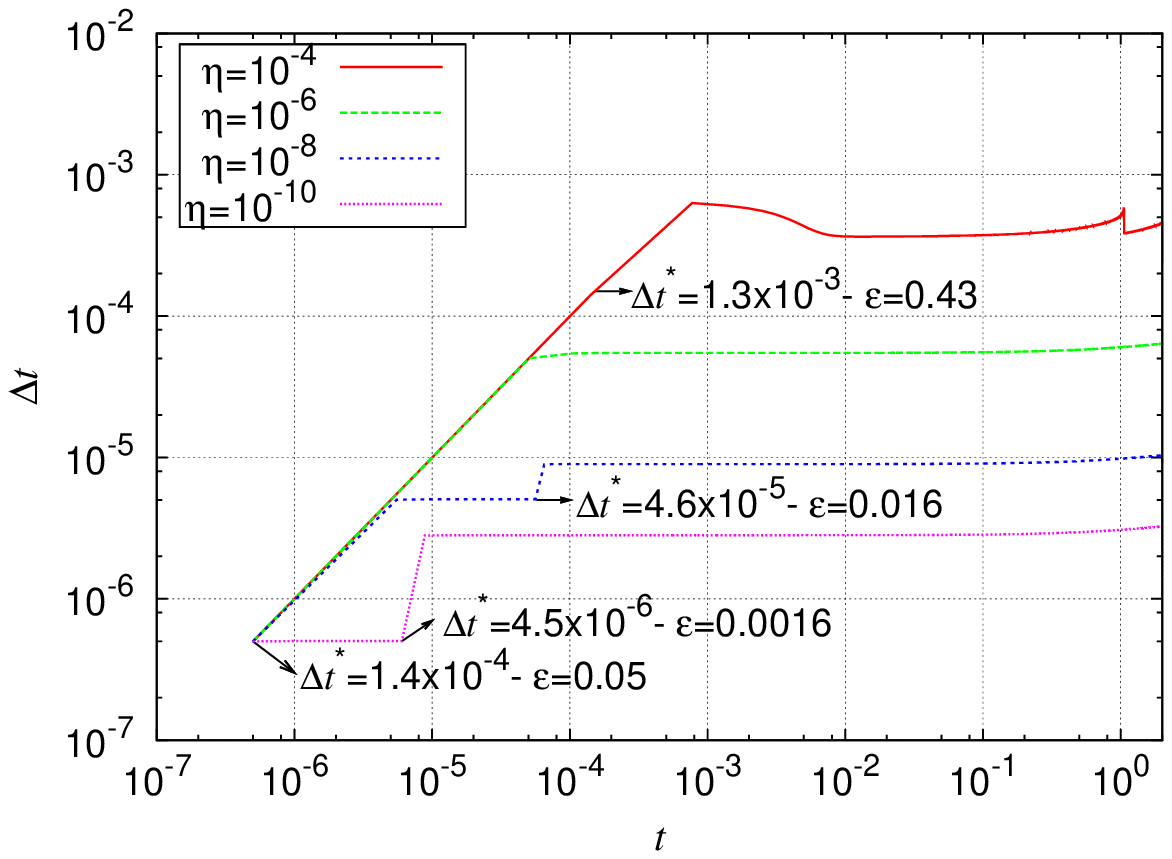}
 \includegraphics[width=0.45\hsize]{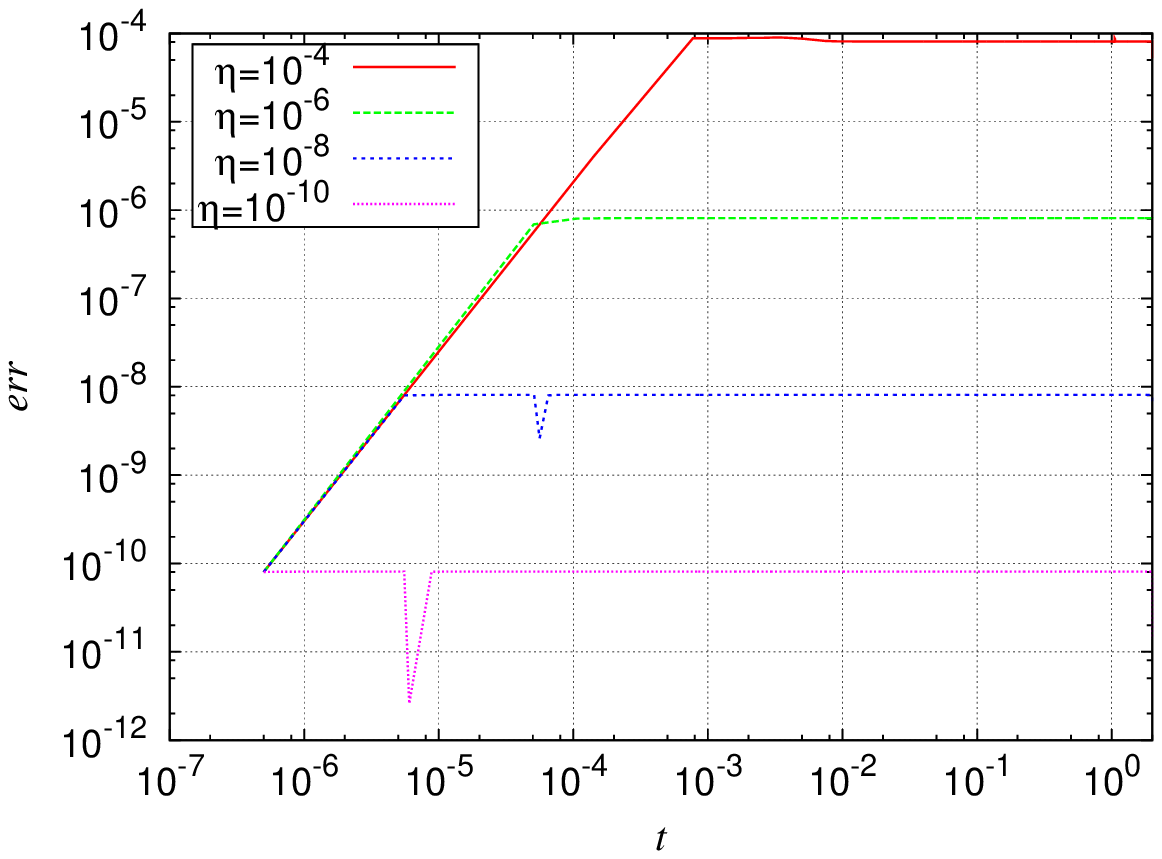} 
   \caption{BZ equation.
Time evolution of
accepted splitting time steps $\Delta t$ (left)
and
local $L^2$ error estimates $err=\|S^{\Delta t}u_0-S^{\Delta t}_{\varepsilon}u_0\|_{L^2}$
(right),
for several tolerances
$\eta$,
considering critical $\Delta t^\star$ and computation of $\varepsilon$.
 \label{bz_step_error_corr}}
 \end{center}
\end{figure}

For $\eta=10^{-6}$, we keep initial $\Delta t^\star$ and $\varepsilon_0$
since
$\Delta t \in [\beta \Delta t^\star, \gamma \Delta t^\star]$
as we can see in Figure \ref{bz_error}.
Finally, for $\eta=10^{-8}$ and $\eta=10^{-10}$,
$\Delta t < \beta \Delta t^\star$ and thus, $\varepsilon$
is recomputed, giving respectively $\varepsilon \approx 0.016$
and $0.0016$.
In particular, we consider larger splitting time steps
for which Strang local errors are better predicted.
Table \ref{bz_err_global_corr} shows that error control is this time guaranteed for all
values of tolerance $\eta$,
and thus, for a larger range of time steps.
Compared with previous results in Table \ref{bz_err_global}, we 
correct completely the errors in the prediction of local error estimates,
which yields more accurate
resolutions for the largest tolerances; 
whereas slightly less accurate results
are obtained for the smallest tolerances since larger splitting time steps are considered.
\begin{table}[!htb]
\caption{BZ equation. 
$L^2$ errors at final time $t=2$ for $a$, $b$, $c$ variables and
several tolerances $\eta$, considering critical $\Delta t^\star$ and computation of $\varepsilon$.}
\label{bz_err_global_corr}
\begin{center}
\begin{tabular}{cccc}
\hline\noalign{\smallskip}
$\eta$& $L^2$ error $a$ & $L^2$ error $b$ & $L^2$ error $c$ \\ 
\noalign{\smallskip}
\hline
\noalign{\smallskip}
$10^{-4}$ & $6.85\times10^{-5}$ & $9.04\times10^{-5}$ & $4.06\times10^{-5}$ \\
$10^{-6}$ & $1.71\times10^{-6}$ & $1.83\times10^{-6}$ & $7.98\times10^{-7}$ \\
$10^{-8}$ & $4.53\times10^{-8}$ & $4.84\times10^{-8}$ & $2.12\times10^{-8}$ \\
$10^{-10}$ & $4.48\times10^{-9}$ & $4.77\times10^{-9}$ & $2.15\times10^{-9}$ \\
\hline
\end{tabular}
\end{center}
\end{table}

\subsection{Simulation of multi-pulsed gas discharges}
In this section, 
we consider a simplified model of 
plasma discharges at atmospheric pressure
for which 
we analyze the performance of the proposed numerical strategy in a
configuration
of nanosecond repetitively pulsed discharges.
This kind of phenomenon is studied
for plasma assisted combustion or flow control, 
for which the enhancement of the gas flow chemistry or momentum transfer 
during typical time scales of the flow of $10^{-4}-10^{-3}$s, 
is due to consecutive discharges generated by high frequency (in the kHz range) sinusoidal or 
pulsed applied voltages \cite{Pilla:2006}.
As a consequence, during the post-discharge phases of the order of tens of microseconds, not only 
time scales are very different from those during discharges of a few tens of nanoseconds,
but a complete different physics is taking place.
Then, to the rapid multi-scale configuration during discharges,
we have to add other rather slower multi-scale phenomena 
in the post-discharge,  
such as recombination of charged species, heavy-species chemistry, diffusion, gas heating and convection.
Therefore,
it is 
very challenging 
to efficiently simulate this kind of highly multi-scale problems
and to accurately describe the physics of the plasma/flow interaction 
between consecutive 
discharge/post-discharge phases. 

General model to study gas discharge dynamics is 
based on the following drift-diffusion equations for electrons and ions, 
coupled with Poisson's equation \cite{Babaeva:1996,Kulikovsky:1997c}:
%%%%%%%%
\def\ne{n_{\rm e}}
\def\np{n_{\rm p}}
\def\nn{n_{\rm n}}
\def\ve{\mathbf{v}_{\! \rm e}}
\def\vp{\mathbf{v}_{\! \rm p}}
\def\vn{\mathbf{v}_{\! \rm n}}
\def\De{D_{\rm e}}
\def\Dp{D_{\rm p}}
\def\Dn{D_{\rm n}}
\def\SeP{\ne\alpha |\mathbf v_{\rm e}|}
\def\SpP{\ne\alpha |\mathbf v_{\rm e}|}
\def\SeM{\ne\eta  |\mathbf v_{\rm e}| +  \ne\np\beta_{\rm ep}}
\def\SpM{\ne\np\beta_{\rm ep} + \nn\np\beta_{\rm np}}
\def\SnM{\nn\np\beta_{\rm np}}
\def\SnP{\ne\eta  |\mathbf v_{\rm e}|}
\def\SD{\nn\gamma}
%%%%%%%%%
\begin{equation}\label{trasp} 
\left.
\begin{array}{rcl}
\partial_t \ne -\partial_ \mathbf{x}\cdot\ne\,\ve 
  -\partial_ \mathbf{x}\cdot(\De\ \partial_ \mathbf{x}\ne) &=& \SeP -\SeM + \SD,\\ 
\partial_t \np +\partial_ \mathbf{x}\cdot\np\vp 
  -\partial_ \mathbf{x}\cdot(\Dp\,\partial_ \mathbf{x}\np) &=& \SpP -\SpM,\\ 
\partial_t \nn -\partial_ \mathbf{x}\cdot\nn\vn 
  -\partial_ \mathbf{x}\cdot(\Dn\,\partial_ \mathbf{x}\nn)  &=& \SnP -\SnM - \SD, 
\end{array}
\right\}
\end{equation}
\begin{equation}
\varepsilon_0\, \partial_ \mathbf{x}^2 V = -q_{\rm e}(\np-\nn-\ne), \label{poisson}
\end{equation}
where 
$\mathbf{x}\in \mathbb{R}^d$,
$n_i$ is the  density of species $i$ (e: electrons, p: positive ions, n: negative ions),
$V$ is the electric potential,
$\mathbf{v}_i= \mu_i \mathbf E$ ($\mathbf E$ being the electric field) is the drift velocity.
$D_i$ and $\mu_i$,  are diffusion coefficient and absolute value of
mobility of charged species $i$,
$q_{\rm e}$ is the absolute value of electron charge, 
and   $\varepsilon_0$  is permittivity of free space.
$\alpha$ is the impact ionization coefficient, $\eta$
stands for electron attachment on neutral molecules, $\beta_{\rm ep}$
and $\beta_{\rm np}$
accounts respectively for electron-positive ion
and 
negative-positive ion recombination,
and $\gamma$ is the detachment coefficient.
Electric field $\mathbf E$ and potential $V$ are related by
\begin{equation}\label{pois2}
\mathbf E = - \partial_ \mathbf{x} V.
\end{equation}

Nevertheless, in this paper, we will consider a simplified reaction-diffusion
1D model
based on (\ref{trasp}):
\begin{equation}\label{trasp_2} 
\left.
\begin{array}{rcl}
\partial_t \ne  %
  -D\, \partial^2_x \ne &=& \SeP -\SeM,\\ 
\partial_t \np 
  -D\, \partial^2_x \np &=& \SpP -\SpM,\\ 
\partial_t \nn
  -D\, \partial^2_x \nn  &=& \SnP -\SnM. 
\end{array}
\right\}
\end{equation}
As in (\ref{trasp}),
all the coefficients of the model are 
functions of the local reduced electric field $E/N_{\rm gas}$,
where $E$ is the electric field magnitude and $N_{\rm gas}$ is the
air neutral density. 
Transport parameters and reaction rates
for air are taken from \cite{Morrow:1997}, with
attachment coefficients taken 
from \cite{Kossyi:1992}.

In this numerical illustration,
we consider an air gap of $0.5\,$cm where we have a high initial
distribution of electrons and ions over the region 
$[0,0.01]\,$cm.
A constant electric field of $\sim 40\,$kV/cm is then applied over this region
during $10\,$ns with a pulse period of $1\,\mu$s.
All parameters in (\ref{trasp_2}) are computed with the imposed field
without solving neither (\ref{poisson}) nor (\ref{pois2}).
Finally, we consider a constant diffusion coefficient: $D=50\,$cm$^2$/s
and a spatial discretization of $1001$ points.
Figure \ref{charges} shows the spatial distribution of electron density
just before and after each pulse.
Globally,
there are at least two completely different physical configurations
given either by high reactive activity whenever the electric field is 
applied, or rather by the propagative nature of the post-discharge phase.
\begin{figure}[!htb]
 \begin{center}
 \includegraphics[width=0.45\hsize]{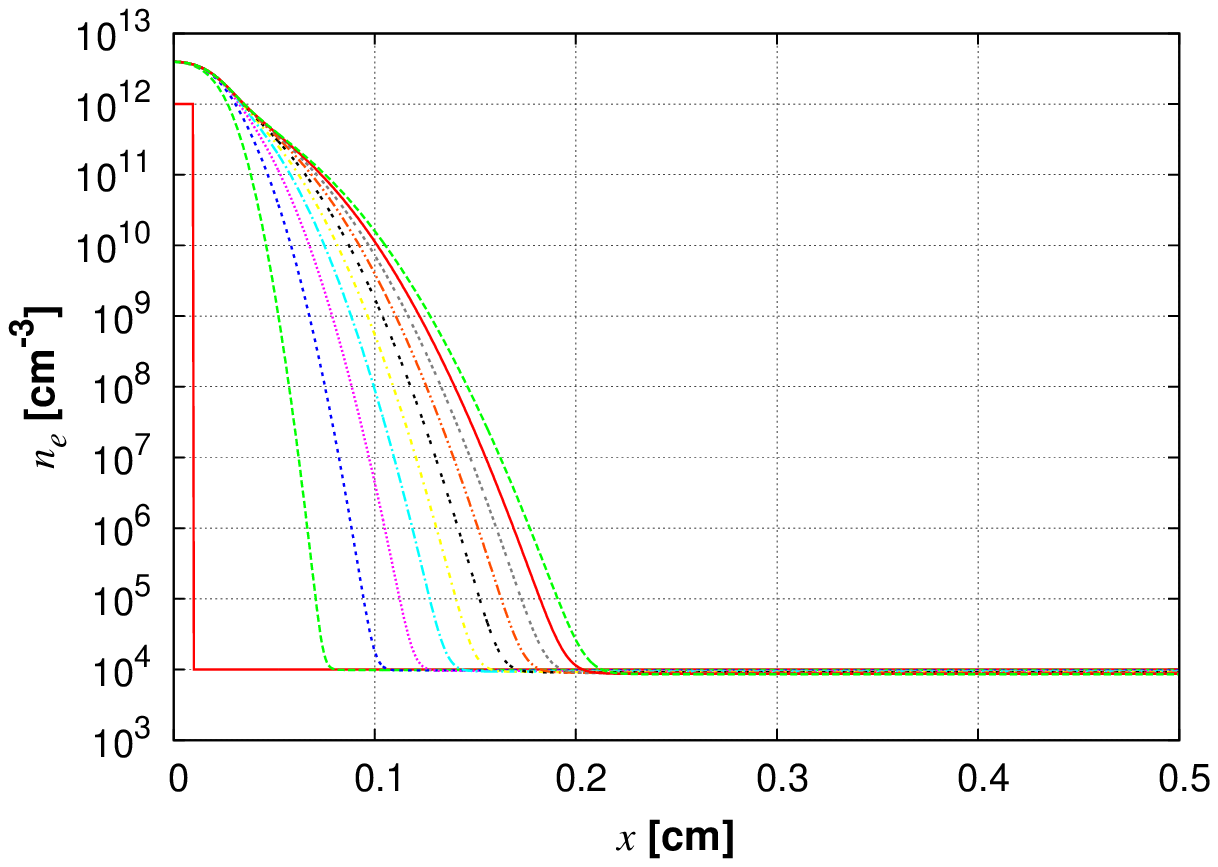}
 \includegraphics[width=0.45\hsize]{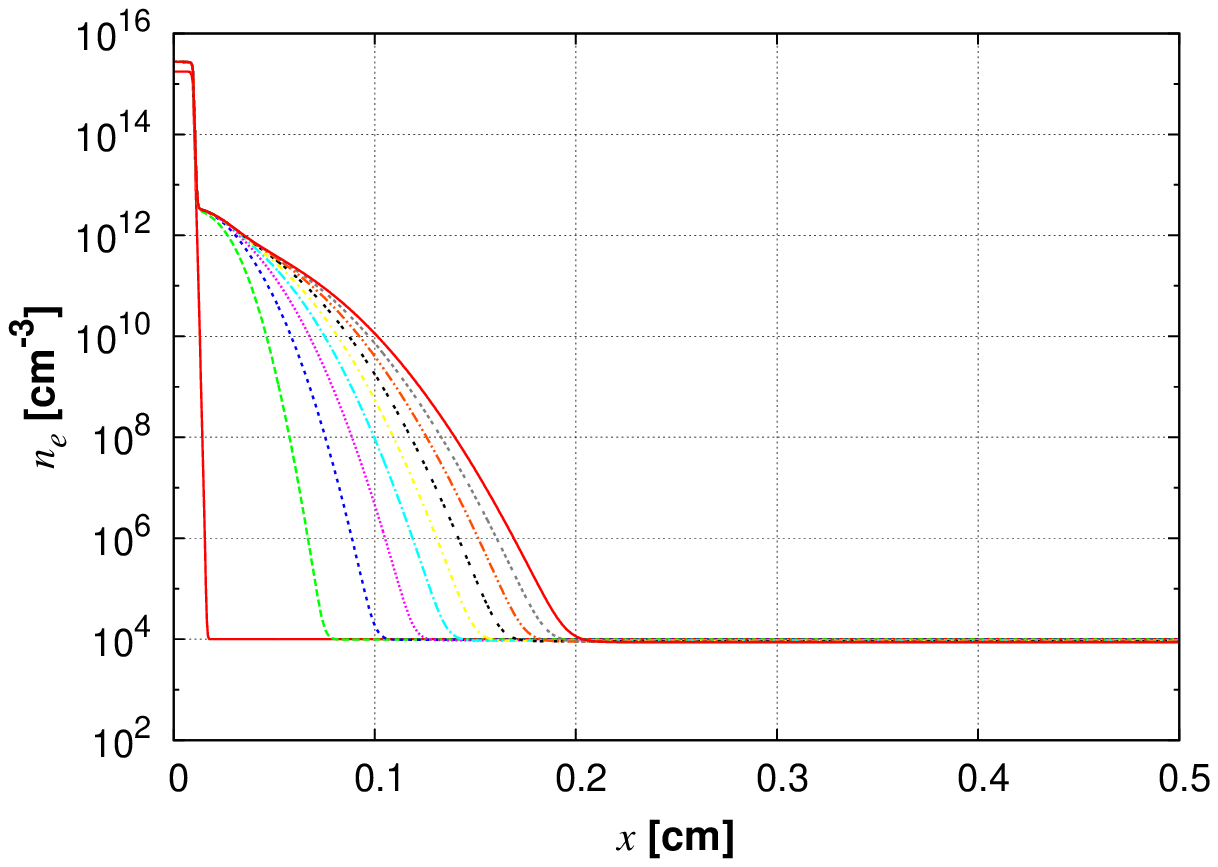} 
   \caption{Repetitive gas discharge model.
Spatial distribution of electron density before (left) and after (right) each pulse,
starting from initial distribution (left) and for a duration of ten pulses.
 \label{charges}}
 \end{center}
\end{figure}

Considering the adaptive strategy described in $\S~$\ref{num_str}
with $\Delta t^0= 10^{-10}$,  $\varepsilon _0=0.05$ and the same
parameters used for the previous BZ simulation,
computation is initialized with a time step included in the pulse
duration.
Figure \ref{charges2} shows the corresponding splitting time steps
for a tolerance of $\eta = 10^{-3}$.
Splitting time step features a periodic behavior and succeed to 
consistently adapt itself to the discharge/post-discharge phenomena.
This yields high varying time steps going from $\sim 10^{-10}$
to $\sim 10^{-7}$.
Therefore,
after each post-discharge phase, since
the new time step is computed based on the previous one
according to (\ref{delta_split_ef}), this new time step
will surely skip the next pulse.
In order to avoid this, each time we get into a new
period, we
initialize time step with the length of the pulse: $\Delta t = 10\,$ns;
this time step is obviously rejected as seen in Figure 
\ref{charges2}, as well as the next ones, until
we are able to retrieve the right dynamics of the phenomenon
for the required accuracy tolerance.
No other intervention is needed neither for modeling parameters
nor for numerical solvers in order to automatically
adapt time step to describe the several time scales of
the phenomenon within a prescribed accuracy.
\begin{figure}[!htb]
 \begin{center}
 \includegraphics[width=0.45\hsize]{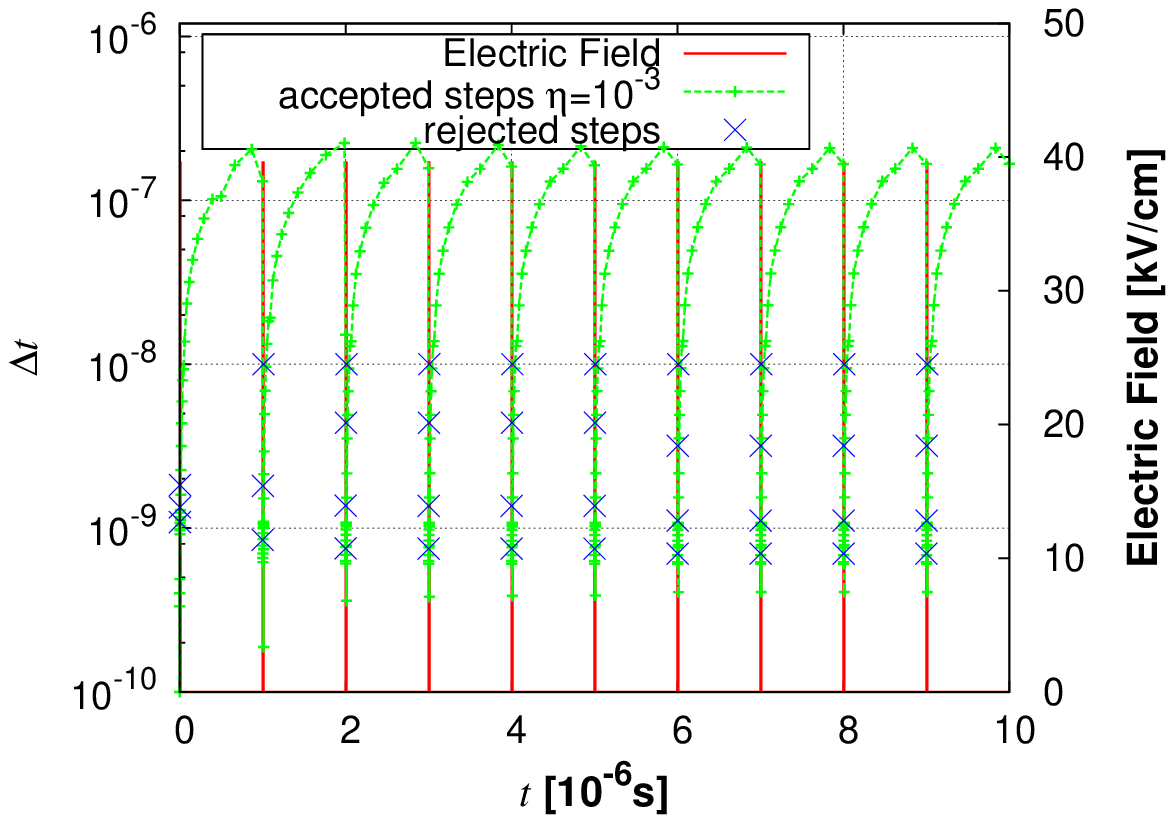}
 \includegraphics[width=0.45\hsize]{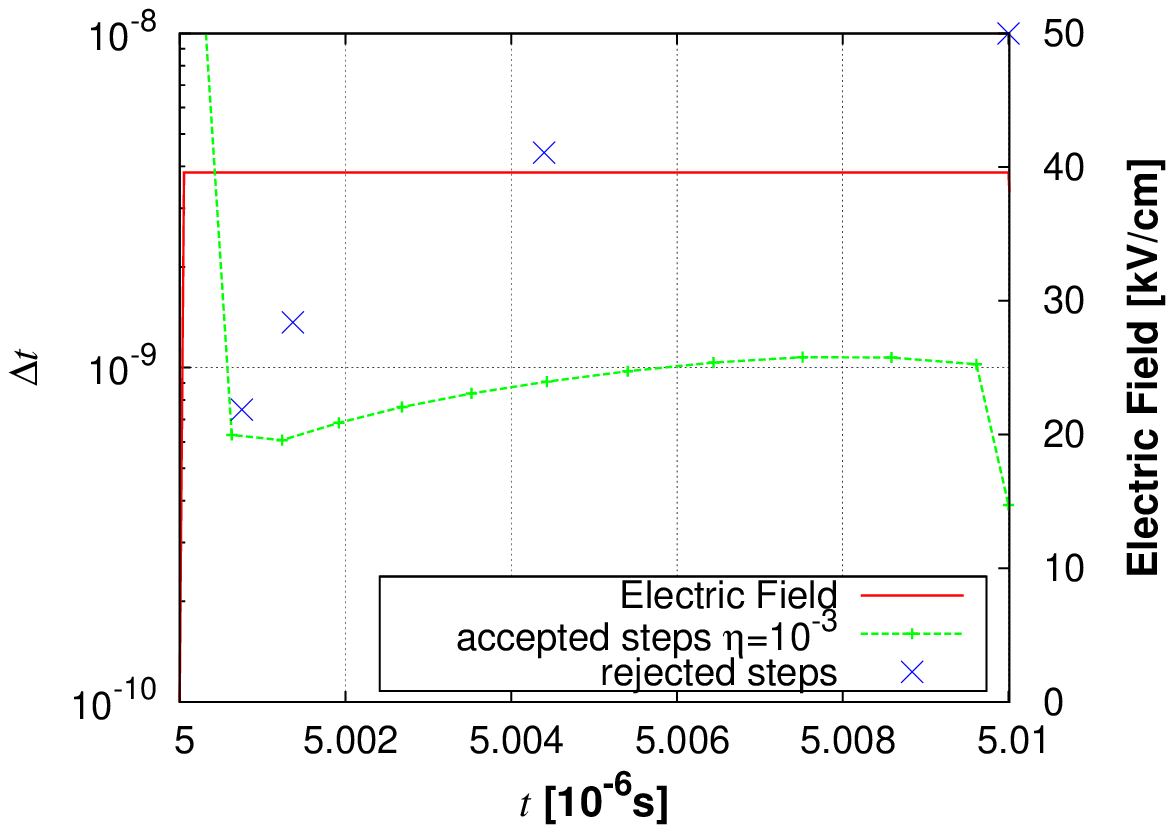}  
 \includegraphics[width=0.45\hsize]{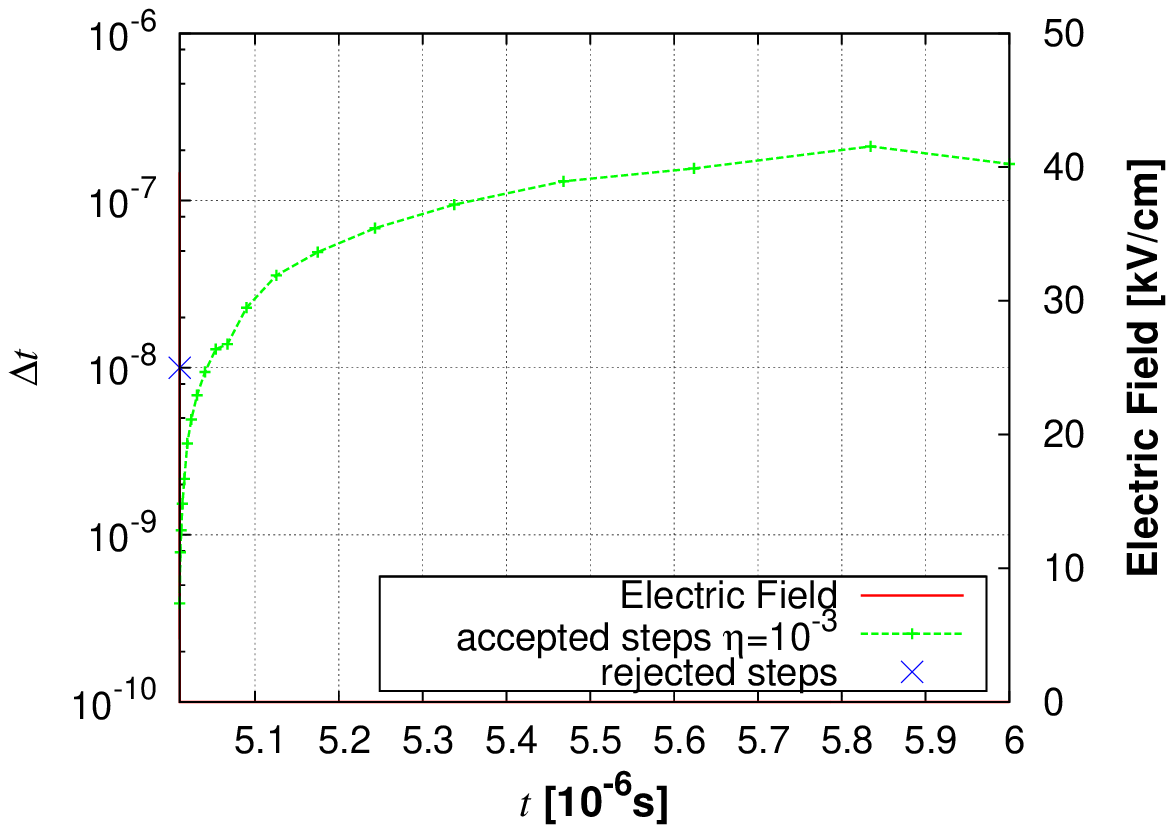}  
 \includegraphics[width=0.45\hsize]{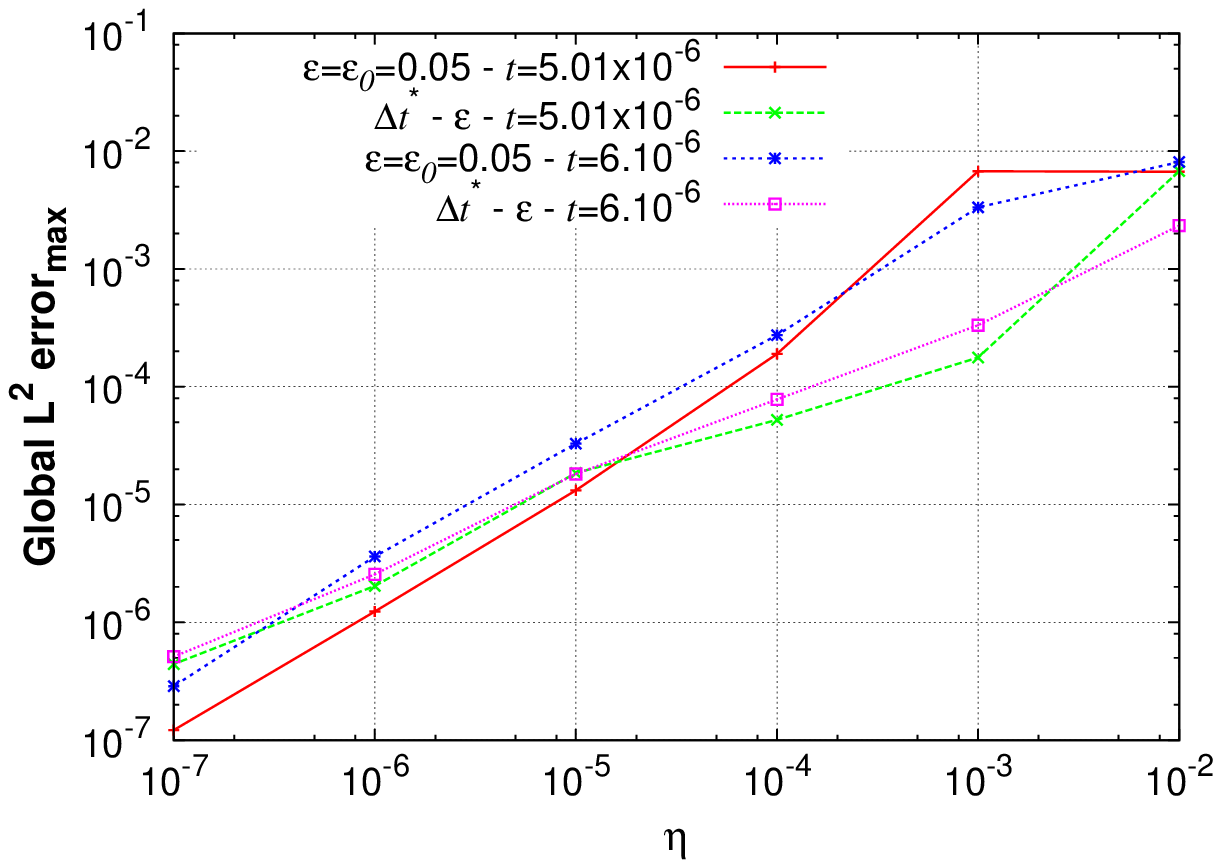}  
   \caption{Repetitive gas discharge model.
Time evolution of accepted and rejected splitting time steps, and imposed electric field
for $t\in[0,10]\,\mu$s (top left), during pulse $t\in[5,5.01]\,\mu$s (top right)
and post-discharge $t\in[5.01,6]\,\mu$s (bottom left).
Bottom right: global $L^2$ errors 
at the end of the pulse ($t=5.01\,\mu$s) and the post-discharge phase ($t=6\,\mu$s)
with and without $\Delta t^\star$ and $\varepsilon$ computation.
 \label{charges2}}
 \end{center}
\end{figure}

For this application, we compute critical $\Delta t^\star$ and possibly $\varepsilon$, for
$N=10$ and $N=100$ in each period in order to perform these computations at least
once during the discharge and post-discharge regimes.
For example, for $t\in [5,6]\, \mu$s as in 
Figure \ref{charges2}, $\varepsilon = \varepsilon _{max}$ with 
$\Delta t^\star \approx 4.3 \times 10^{-9}$ during the pulse, and
$\varepsilon \approx 0.26$ with 
$\Delta t^\star \approx 1.6 \times 10^{-7}$ for the rest of the period.
Similar values are found for the other periods.
Notice that after each pulse, $\Delta t^\star$ is automatically updated 
because $\Delta t$ increases and then $\Delta t$ gets equal to $\Delta t^\star$.
In particular, the important difference between $\Delta t^\star$
for each region, comes naturally from the completely different modeling
parameters and hence, physics description of each regime.

An effective error control is achieved for each part of the phenomenon,
as we can deduce from the global error between splitting and reference 
solutions at the end of the pulse ($t=5.01\,\mu$s) and at the end
of the post-discharge phase ($t=6\,\mu$s).
If we compare these results with the ones obtained without estimating
neither $\Delta t^\star$ nor $\varepsilon$ with 
$\varepsilon=\varepsilon _0=0.05$, we can draw the same conclusions
as in the BZ application.
For less accurate resolutions 
with high tolerances, the proposed strategy corrects the error in the
local error estimates made with $\varepsilon=\varepsilon _0=0.05$;
in particular, for $\eta =10^{-3}$ there is a ratio of about
$10$ between both solutions.
For higher tolerances, $\eta \geq 10^{-2}$, both methods yield
a time step equal to the pulse duration, $\Delta t = 10\,$ns.
On the other hand, 
for the smallest
tolerances,
slightly more accurate solutions are obtained  
with a fixed $\varepsilon=\varepsilon _0$
because smaller splitting time steps are used.

\FloatBarrier

\section{Conclusions}
The present work proposes a new resolution
strategy for stiff evolutionary PDEs
based
on an efficient
splitting scheme previously developed
\cite{article_avc,article_mr}
that considers high order dedicated
integration methods for each subproblem in 
order to properly solve the fastest time
scales associated with each one of them,
and in such a way that the main source of error is led
by the operator splitting error.
Then, to control the error of the resolution, 
it relies
on an adaptive splitting time 
technique that allows to discriminate
the global time scales related to the coupled phenomenon,
given a required level of accuracy of computations.
Compared with a standard procedure for which accuracy is guaranteed
by considering time steps of the order of the fastest scale,
the error control featured by our method  
implies 
an effective accurate resolution 
for problems modeling various physical scenarios,
independent of the
fastest physical time scale,
and an important improvement of
computational efficiency whenever highly unsteady
phenomena is simulated.
In particular, we have successfully applied the proposed
strategy to a simplified model of 
plasma discharges that nevertheless exhibits a broad time scale
spectrum coming from the modeling equations and also 
important and discontinuous variation of parameters
in time and in space that notably increase the
numerical complexity of the problem.

A numerical analysis of the method has been developed in order
to settle a solid mathematical background, and a
complementary numerical procedure was conceived in order
to overcome classical restrictions of adaptive time stepping
schemes whenever asymptotic estimates fail to predict the
dynamics of the problem.
A both mathematical and numerical detailed study of the method
has thus led to a fully complete adaptive time stepping strategy
that guarantees an effective control of the errors of integration
for a large range of time steps;
a key issue for problems for which splitting time steps can
go beyond the fastest physical scales of the problem.
The contribution of this paper is then mainly given 
by a dedicated adaptive time splitting method
for stiff PDEs, and by a complete study 
of the behavior of time stepping schemes based on lower order embedded methods,
for the whole 
set of potential time steps.
In this paper we have always considered fine enough spatial discretizations
in order to perform an evaluation of the theoretical estimates introduced for
the proposed time integration scheme. For higher dimensional problems,
fine spatial discretization becomes a critical issue in terms of computational
costs and a technique of local grid refinement might be a good solution to
guarantee the theoretical behavior of the splitting schemes (see for instance
\cite{article_mr}). Nevertheless, a mathematical study on the splitting errors
with discretized operators will certainly be an useful tool to yet improve the
performance of these techniques. This and other related theoretical aspects
are particular topics of our current research.

\section*{Acknowledgments}
This research was 
supported by a fundamental project grant from ANR (French National Research Agency - ANR Blancs):
\emph{S\'echelles} (project leader S. Descombes),
and by a Ph.D. grant for 
M. Duarte 
from Mathematics (INSMI) and Engineering (INSIS) Institutes of CNRS 
and supported by INCA project (National Initiative for Advanced Combustion - CNRS - ONERA - SAFRAN).

%%-----------------------------
%%      your bibliography
\bibliographystyle{plain}
\bibliography{conf_biblio.bib}

%%-----------------------------

\end{document}